\documentclass[12pt]{article}

\usepackage[all,cmtip]{xy}
\usepackage{hyperref}
\hypersetup{colorlinks = true, allcolors = blue}
\usepackage{fullpage}

\usepackage{mathtools}
\usepackage{extarrows}

\usepackage[utf8]{inputenc}
\usepackage{setspace}
\usepackage{tocloft}

\usepackage{graphicx, titlesec}
\usepackage{amscd,amsmath,amsthm,amssymb,bm}%\usepackage{amsmath,bm}

\usepackage{verbatim}
\usepackage[all]{xy}
\usepackage{color}
\usepackage{tikz, float} \usetikzlibrary {positioning}
\usetikzlibrary{calc}

\usepackage{makeidx}         % allows index generation
\usepackage{graphicx}        % standard LaTeX graphics tool
\usepackage{multicol}        % used for the two-column index
\usepackage[bottom]{footmisc}% places footnotes at page bottom

\usepackage{array}

\usepackage{newtxmath}      % provides definitions for: \jmath, \amalg, \coprod
\usepackage{theoremref}
\usepackage{url}

\makeindex

\newcommand{\id}{\text{id}}
\newcommand{\ini}{\text{in}}

\newcommand{\MT}{\mathcal T}

\newcommand{\MCM}{\mathcal M}
\newcommand{\KK}{\mathbb{K}}

\titleformat{\subsection}[runin]
{\normalfont\bfseries}{\thesubsection}{.5em}{}
\def\NZQ{\mathbb}               % the font for N,Z,Q,R,C

\def\PP{{\NZQ P}}

\def\frk{\mathfrak}               % font for "Fraktur"

\def\Phi{{\frk n}}
\def\Phi{{\frk N}}

\def\A{{\mathcal A}}

\def\opn#1#2{\def#1{\operatorname{#2}}} % to make operators
\opn\chara{char} \opn\length{\ell} \opn\pd{pd} \opn\rk{rk}
\opn\projdim{proj\,dim} \opn\injdim{inj\,dim} \opn\rank{rank}
\opn\depth{depth} \opn\grade{grade} \opn\height{height}
\opn\embdim{emb\,dim} \opn\codim{codim}

\opn\Tr{Tr} \opn\bigrank{big\,rank}
\opn\superheight{superheight}\opn\lcm{lcm}
\opn\trdeg{tr\,deg}%\emph{
\opn\reg{reg} \opn\lreg{lreg} \opn\ini{in} \opn\lpd{lpd}
\opn\size{size} \opn\sdepth{sdepth}
\opn\link{link}\opn\fdepth{fdepth}\opn\lex{lex}
\opn\LM{LM}
\opn\LC{LC}
\opn\NF{NF}
\opn\Merge{Merge}
\opn\sgn{sgn}
\opn\type{type}
\opn\div{div} \opn\Div{Div} \opn\cl{cl} \opn\Pic{Pic}
\opn\Prin{Prin}
\opn\op{op}
\opn\indeg{indeg} \opn\outdeg{outdeg}
\opn\red{red}
\opn\Spec{Spec} \opn\Supp{Supp} \opn\supp{supp} \opn\Sing{Sing}
\opn\Ass{Ass} \opn\Min{Min}\opn\Mon{Mon} \opn\val{val}
\opn\Ann{Ann} \opn\Rad{Rad} \opn\Soc{Soc}
 \opn\Ker{Ker} \opn\Coker{Coker} \opn\Am{Am}
\opn\Hom{Hom} \opn\Tor{Tor} \opn\Ext{Ext} \opn\End{End}
\opn\Aut{Aut} \opn\id{id}

\opn\nat{nat}
\opn\pff{pf}%   \pf exists already
\opn\Pf{Pf} \opn\GL{GL} \opn\SL{SL} \opn\mod{mod} \opn\ord{ord}
\opn\Gin{Gin} \opn\Hilb{Hilb}\opn\sort{sort}
\opn\Image{Image}
\opn\vol{Vol}
\opn\aff{aff} \opn\con{conv} \opn\relint{relint} \opn\st{st}
\opn\lk{lk} \opn\cn{cn} \opn\core{core} \opn\vol{vol}
\opn\link{link} \opn\star{star}\opn\lex{lex}\opn\set{set}
\opn\dist{dist}
\opn\gr{gr}

\def\pot#1#2{#1[\kern-0.28ex[#2]\kern-0.28ex]}

\opn\dirlim{\underrightarrow{\lim}}
\opn\inivlim{\underleftarrow{\lim}}
\let\to=\rightarrow

\def\Implies{\ifmmode\Longrightarrow \else
        \unskip${}\Longrightarrow{}$\ignorespaces\fi}
\def\implies{\ifmmode\Rightarrow \else
        \unskip${}\Rightarrow{}$\ignorespaces\fi}
\def\iff{\ifmmode\Longleftrightarrow \else
        \unskip${}\Longleftrightarrow{}$\ignorespaces\fi}

\let\:=\colon

\newtheorem{theorem}{Theorem}[section]
\newtheorem{lemma}[theorem]{Lemma}

\newtheorem{proposition}[theorem]{Proposition}

\theoremstyle{remark}
\newtheorem{remark}[theorem]{Remark}

\theoremstyle{definition}
\newtheorem{example}[theorem]{Example}

\newtheorem{definition}[theorem]{Definition}
\newtheorem{Notation}{Notation}[section]
\newtheorem*{theorem*}{Theorem}
\newtheorem*{corollary*}{Corollary}

\DeclareMathOperator{\Flag}{Fl}

\let\kappa=\varkappa

\def\qed{\ifhmode\textqed\fi
      \ifmmode\ifinner\quad\qedsymbol\else\dispqed\fi\fi}
\def\textqed{\unskip\nobreak\penalty50
       \hskip2em\hbox{}\nobreak\hfil\qedsymbol
       \parfillskip=0pt \finalhyphendemerits=0}
\def\dispqed{\rlap{\qquad\qedsymbol}}

\opn\dis{dis}
\def\pnt{{\raise0.5mm\hbox{\large\bf.}}}

\opn\Lex{Lex}
\opn\syz{{\rm syz}}
\opn\spoly{{\rm spoly}}
\opn\LM{{\rm LM}}
\opn\lm{{\rm lm}}
\opn\lcm{{\rm lcm}} \opn\A{\mathcal A}

\numberwithin{equation}{section}

\newcommand{\ul}[1]{\underline{#1}}

\DeclareMathOperator{\init}{in}

\usepackage{multirow}
\hypersetup{
  colorlinks=true,
}
\begin{document}
%\begin{frontmatter}

\title{Standard monomial theory and toric degenerations of\\
Richardson varieties in flag varieties}

\author{
Narasimha Chary Bonala, Oliver Clarke and Fatemeh Mohammadi
}

\maketitle

\begin{abstract}
We study standard monomial bases for Richardson varieties inside the flag variety. In general, writing down a standard monomial basis for a Richardson variety can be challenging, as it involves computing so-called defining chains or key tableaux. However, for a certain family of Richardson varieties, indexed by compatible permutations, we provide a very direct and straightforward combinatorial rule for writing down a standard monomial basis. We apply this result to the study of toric degenerations of Richardson varieties. In particular, we provide a new family of toric degenerations of 
Richardson varieties inside flag varieties.
\end{abstract}
%\end{frontmatter}

\vspace{-7mm}

{\hypersetup{linkcolor=black}\setcounter{tocdepth}{1}\setlength\cftbeforesecskip{.3pt}{\tableofcontents}}

\section{Introduction}
The geometry of the flag variety heavily depends on the study of its Schubert varieties. For example, they provide an excellent way of understanding the multiplicative structure of the cohomology ring of the flag variety. In this context, it is essential to understand how Schubert varieties intersect in a general position. 
A Richardson variety in the flag variety is the intersection of a Schubert variety and an opposite Schubert variety.  In \cite{deodhar1985some} and  \cite{richardson1992intersections}, the fundamental properties of these varieties are studied, including their irreducibility.  
Many geometric properties of the flag variety and its subvarieties can be understood through standard monomial theory (SMT). For example, vanishing results of cohomology groups, normality and other singularities of Schubert varieties, see e.g.~\cite[Chapter 3]{seshadri2016introduction}. Also, the relationship between $K$-theory of the flag variety and SMT is established in \cite{lakshmibai2003richardson}.

Let $\mathbb{K}[P_J]$ be the polynomial ring on the Pl\"ucker variables $P_J$ for non-empty subsets $J$ of $\{1, \ldots, n\}$ and let $I_n\subset \mathbb{K}[P_J]$ be the Pl\"ucker ideal of the flag variety $\Flag_n$. 
The homogeneous coordinate ring of 
$\Flag_n$ is given by $\mathbb{K}[P_J]/I_{n}$. We say that a monomial $P = P_{J_1}\dots P_{J_d}$ is \textit{standard for $\Flag_n$} if $J_1 \le \dots \le J_d$.
A standard monomial basis of $\mathbb{K}[P_J]/I_{n}$ is a subset of standard monomials that forms a basis for  $\mathbb{K}[P_J]/I_{n}$ as a vector space.
Hodge in \cite{hodge1943some} provided a combinatorial rule to choose such basis for Grassmannians in terms of semi-standard Young tableaux. He also proved that such a basis is compatible with Schubert varieties. More precisely, the basis elements that remain non-zero after restriction form a basis for the quotient ring associated to the Schubert varieties. 
Hodge's work is generalised to flag varieties by Lakshmibai, Musili and Seshadri, see e.g.~\cite{seshadri2016introduction}
for a more detailed exposition.

In this paper, we investigate when the standard monomials directly restrict to a monomial basis for the Richardson variety, thus providing a particularly simple rule for determining standard monomial bases for particular Richardson varieties. A standard monomial $P = P_{J_1} \dots P_{J_d}$ \textit{restricts} to a non-zero function on the Richardson variety $X_w^v$ if and only if $v \le J_i \le w$ for all $i$. We say that such a monomial restricts to the Richardson variety $X_w^v$ and we often write these restricted monomials as semi-standard Young tableaux, see Section~\ref{sec:prelim_tableaux}.
For Richardson varieties in the Grassmannians, the restricted standard monomials always form a monomial basis. However, this is not true for an arbitrary Richardson variety in the flag variety, see Example~\ref{example1}. And so the conventional method for determining a monomial basis for the coordinate ring of a Richardson variety inside the flag variety is fairly complicated, see Section~\ref{sec:standard_monomial}. Other combinatorial methods for calculating these monomials have been explored  using so-called \textit{key tableaux} \cite{willis2011tableau}.
We restrict our attention to the family of Richardson varieties $X_w^v$, where $(v,w) \in \MT_n$ from \eqref{eq:def_TN}. 
Our main result is the following:

\begin{theorem*}[Theorem~\ref{thm:ssyt_fl_diag}]\label{thm:standardmonomials}
The restriction of standard monomials for the flag variety forms a standard monomial basis for the Richardson variety $X_w^v$, for all $(v,w) \in \MT_n$ from \eqref{eq:def_TN}.
\end{theorem*}

We also observe a surprising relation between the pairs in $\mathcal {T}_n$ and toric degenerations of Richardson varieties.
A toric degeneration of a variety $X$ is a flat family $f: \mathcal X \to  \mathbb A^{\!{1}}$, where the special fiber (over zero) is a toric variety and all other fibers are isomorphic to $X$. In particular, some of the algebraic invariants of $X$, such as the Hilbert polynomial, will be the same for all the fibers. Hence, we can do the computations on the toric fiber.
The study of toric degenerations of flag varieties was started in \cite{gonciulea1996degenerations} by Gonciulea and Lakshmibai using
standard monomial theory. In \cite{KOGAN}, Kogan and Miller obtained toric degenerations of flag varieties using geometric methods. Moreover, in \cite{caldero2002toric} Caldero constructed such degenerations using tools from representation theory.
In \cite{kim2015richardson}, Kim studied the Gr{\"o}bner degenerations of Richardson varieties inside the flag variety, where the special fiber is the toric variety of the Gelfand-Tsetlin polytope; this is a generalisation of the results of \cite{KOGAN}. 
 We notice that the corresponding ideals of many such degenerations 
contain monomials and so their corresponding varieties are not toric.
Hence, we aim to characterise 
such toric ideals.
In particular, we explicitly describe degenerations of 
Richardson varieties inside 
flag varieties, and provide a complete characterisation for permutations leading to 
monomial-free ideals.

\begin{theorem*}[Theorem~\ref{thm:toric_degen}]
Every pair of permutations $(v,w)$ in $\mathcal{T}_n$ from \eqref{eq:def_TN} gives rise to a toric degeneration of the Richardson variety $X_w^v$ in the flag variety.
\end{theorem*}

The toric degenerations that we construct are Gr\"obner degenerations of the ideal of the Richardson variety with respect to a weight vector which we define in Sections~\ref{subsec:flag} and \ref{subsec:X}. 
Our goal is to study the initial ideal $\init(I(X_w^v))$ for each Richardson variety $X_w^v$ with respect to this weight vector. To do this, we consider the result for the flag variety which is proved in \cite{OllieFatemeh2, miller2004combinatorial}. More precisely, the corresponding initial ideal is \textit{toric} (binomial and prime), see Theorem~\ref{thm:Pure}. We write the initial of the Pl\"ucker ideal as $\init(I_n) = \ker(\phi_n)$, which is equal to the kernel of a monomial map $\phi_n$, see \eqref{eqn:monomialmapflags}. For each Richardson variety $X_w^v$ we construct the \textit{restriction} of $\init(I_n)$ to the variables $\{P_J : J \in T_w^v \}$, which are the variables which do not vanish on the Richardson variety. We can readily obtain the generators of the restriction $\init(I_n)|_{T_w^v}$ using Lemma~\ref{lem:elim_ideal_gen_set}, in particular this ideal is generated by degree two monomials and binomials. 
Our degeneration is known as the \textit{diagonal} Gr\"obner degeneration. In Remark~\ref{sec:compareKim's}, we consider the degenerations in \cite{kim2015richardson} which correspond to the \textit{antidiagonal} Gr\"obner degeneration. We show that our methods can be used to generalise these results and find further toric degenerations.

\medskip

We have summarised our approach to understanding the relationships between the aforementioned ideals, namely in$(I_n)|_{T_w^v}$, in$(I(X_w^v))$ and ker$(\phi_n|_w^v)$ in the following diagram.
\begin{equation}\label{diagram:intro_commutative}
\begin{aligned}
\xymatrixcolsep{3.5pc}
 \xymatrix{
    {I_n} 
    \ar@{~>}[r]^-{\text{initial ideal}}
    \ar@{~>}[d]_{\text{restriction}} &
    {\init(I_n)} 
    \ar@{~>}[d]^{\text{restriction}}
    \ar@2{-}[r]^-{\rm Theorem~\ref{thm:Pure}} &
    {\ker(\phi_n) 
    \subseteq
    \KK[P_J : J \subseteq [n]] 
    \xlongrightarrow{\phi_n}
    \KK[x_{i,j}]}
    \ar@{~>}[d]^-{\text{restriction of $\phi_n$}}
    \\
    {I(X_w^v)} 
    \ar@{~>}[r]_-{\text{initial ideal}} 
    & 
    {\init(I_n)|_{T_w^v}}
    \ar@2{-}[r]^-{\textcolor{blue}{{\bm ?}}} &
    {\ker(\phi_n|_w^v)
    \subseteq
    \KK[P_J : J \in T^v_{w}] 
    \xlongrightarrow{\phi_n|_w^v}
    \KK[x_{i,j}]} \\
    }
\end{aligned}
\end{equation}

Our method is to consider all possible pairs of permutations $(v,w)$ and determine whether the equality labelled by `{\bf ?}' holds. We give the following complete classification:
\begin{theorem*}[Theorem~\ref{thm:toric_fl_rich_diag}]
The ideal $\init(I_n)|_{T_w^v}$ is monomial-free if and only if $(v,w) \in \MT_n$. In particular, if $(v,w) \in \MT_n$ then $\init(I_n)|_{T_w^v}$ is a toric ideal and coincides with the kernel of $\phi_n|_w^v$.
\end{theorem*}
To prove this theorem, we exploit the explicit description of a generating set for $\init(I_n)|_{T_w^v}$ in Section~\ref{subsec:X} and the inductive structural results for elements of $\MT_n$ in Section~\ref{sec:Tn_block_structure}.
In order to determine whether we obtain a toric degeneration of the Richardson variety, we check whether the left hand square in \eqref{diagram:intro_commutative} commutes, i.e.~we check whether $\init(I(X_w^v)) = \init(I_n)|_{T_w^v}$. In Theorem~\ref{thm:toric_degen} we show that the square does indeed commute whenever the standard monomials for the flag variety restrict to a monomial basis for the coordinate ring for the Richardson variety, i.e.~whenever $(v,w) \in \MT_n$ and our main theorem holds. In these cases we obtain toric degenerations of Richardson varieties.

\subsection{Outline of the Paper.} 

In Section~\ref{sec:prim}, 
we fix our notation throughout the paper and we recall the notion of Richardson varieties and Gr\"obner degenerations. 
In Section~\ref{sec:SMT}, we study standard monomial bases of Richardson varieties. We first construct the set $\MT_n$ and introduce the notion of a block structure for its elements. This is our main tool to prove Theorem~\ref{thm:ssyt_fl_diag}.
In Section~\ref{sec:mon_free_ideals}, we classify monomial-free ideals of from 
$\init(I_n)|_{T_w^v}$, see Theorem~\ref{thm:toric_fl_rich_diag}. Section~\ref{sec:toric} contains the proofs of our results on toric degenerations of Richardson varieties, in particular Theorem~\ref{thm:toric_degen}.
In Remark~\ref{sec:compareKim's} we perform calculations for all Richardson varieties in
$\Flag_4$. We 
show how our methods can be used to study initial ideals with respect to different weight vectors. In particular, we outline how our method can refine the results from \cite{kim2015richardson}.

\section{Preliminaries}\label{sec:prim}
\vspace{-2mm}
Throughout we fix an algebraically closed field $\mathbb{K}$ and write $[n]$ for the set $\{1, \dots, n \}$. We denote the symmetric group on $n$ symbols by $S_n$ and for any $w\in S_n$ we write $w = (w_1, \dots, w_n)$, where $w_i = w(i)$ for each $i\in[n]$. 
We fix 
$w_0:=(n, n-1, \ldots, 2, 1)$ for the longest product of adjacent transpositions in $S_n$. The permutations of $S_n$ act naturally on the left of subsets of $[n]$. So, for each $I = \{i_1, \ldots , i_k \} \subset [n]$, we have $w_0 I = \{n+1-i_1, \ldots , n+1-i_k \}$ which is obtained by applying the permutation $w_0$ element-wise to $I$. 
We use $\le$ for the natural partial order on the subsets of $[n]$ given by
\[
\{i_1 < \dots < i_s \} \le \{j_1 < \dots < j_t \} 
\textrm{ if } 
s \ge t \textrm{ and }
i_1 \le j_1, \dots, i_t \le j_t.
\]We recall the \textit{Bruhat order} on $S_n$, which is given by
\[
(v_1, \dots, v_n) \le (w_1, \dots, w_n) \textrm{ if } \{v_1, \dots, v_k \} \le \{w_1, \dots, w_k \} \textrm{ for all } k \in [n].
\]
It is also convenient  for us to define a comparison operator between subsets and permutations:
\[
\{i_1, \dots, i_k \} \le (w_1, \dots, w_n)
\textrm{ if } \{i_1, \dots, i_k \} \le \{ w_1, \dots, w_k \},
\]
\[
(v_1, \dots, v_n) \le \{i_1, \dots, i_k \}
\textrm{ if } \{ v_1, \dots, v_k \} \le \{i_1, \dots, i_k \}.
\]

\begin{remark}
The comparison $\le$ between subsets and permutations can be phrased purely in terms of the Bruhat order as follows. For each subset $I = \{i_1 < \dots < i_s \} \subseteq [n]$, let $\{j_1 < \dots < j_{n-s} \} = [n] \backslash I$ denote its complement.
Then for any pair of permutations $v, w \in S_n$ we have:
\[
I \le w \iff (i_1, i_2 \dots, i_s, j_1, \dots, j_{n-s}) \le w
\quad \textrm{ and } \quad
v \le I \iff v \le (i_s, i_{s-1} \dots, i_1, j_{n-s}, \dots, j_1).
\]
\end{remark}

\vspace{-2mm}\subsection{Flag varieties.}\label{sec:flag_def}\label{sec:prelim_ideals}
A full flag is a sequence of vector subspaces of $\mathbb{K}^n$: $$\{0\}= V_0\subset V_1\subset\cdots\subset V_{n-1}\subset V_n=\mathbb{K}^n$$ where ${\rm dim}_{\mathbb{K}}(V_i) = i$. 
The set of all full flags is called the flag variety and denoted by $\Flag_n$, which is naturally embedded  in a product of Grassmannians. Here, we consider the structure of algebraic variety on $\Flag_n$ induced from the product of Grassmannians.
We view the full flag variety $\Flag_n$ as a homogeneous space for the group SL$(n,\mathbb K)$ of complex $n\times n$ matrices with determinant one. Precisely, there is a natural transitive action of SL$(n, \mathbb K)$ on the flag variety $\Flag_n$ which identifies the variety $\Flag_n$ with the set of left cosets SL$(n, \mathbb K)/B$, where $B$ is the stabiliser of the standard flag
$0\subset \langle e_1\rangle \subset\cdots \subset \langle e_1, \ldots, e_n\rangle=\mathbb K^n $. Here, note that
$B$ is the subgroup of SL$(n,\mathbb K)$ consisting of upper triangular matrices. Given a permutation $w \in S_n$, we denote by $\sigma_w$ the $n\times n$ permutation matrix with $1$'s in positions $(w(i),i)$ for all $i$. By the Bruhat decomposition, we can write the aforementioned set of cosets as 
$$\Flag_n={\rm SL}(n,\mathbb K)/B= \coprod_{w\in S_n}B\sigma_wB/B.$$
The spaces $B\sigma_wB/B$ are all affine and are called Bruhat cells. 
Similarly, for the subgroup of lower triangular matrices $B^-$, the homogeneous space $\Flag_n$ can be decomposed as 
$$\Flag_n=\coprod_{v\in S_n}B^-\sigma_vB/B.$$

\subsection{Richardson varieties.}\label{sec:Richardson}
Let $v, w\in S_n$.
We define the Richardson variety $X_w^v$
associated to $v,w$ as the intersection of Schubert variety $X_w$ and opposite Schubert variety $X^v$ inside the flag variety $\Flag_n$. More precisely, the Schubert and opposite Schubert varieties are defined as the Zariski closure of the corresponding cells in the aforementioned decomposition, namely:
$$X_w=\overline{B\sigma_wB/B} \subseteq \Flag_n\quad\text{and}\quad X^v=\overline{B^-\sigma_vB/B}\subseteq \Flag_n.$$ 
Note that $X_w^v$ is nonempty if and only if $v \leq w$ with respect to the \textit{Bruhat order}, see Section~\ref{sec:prelim_tableaux}. Moreover, the dimension of $X_w^v$ is given by $\dim(X_w^v) = N(w) - N(v)$, where $N(w)$ is the \textit{inversion number} of $w$, i.e.~the total number of pairs $(i,j) \in [n]\times [n]$ such that $i < j$ and $w(i) > w(j)$.
which we denote by $N(w)$.

We also note that the opposite Schubert variety $X^v$ can be observed as a translate $w_0X_{w_0v}$ of the Schubert variety $X_{w_0v}$ since $B^-=\sigma_{w_0}B\sigma_{w_0}$. Moreover, $X_w^{id}=X_w$ and $X_{w_0}^v=X^v$.

\vspace{-2mm}
\subsection{Ideals of flag varieties and Richardson varieties.}\label{sec:ideal}
Every point in the flag variety $\Flag_n$ can be represented by an $(n-1)\times n$ matrix $X=(x_{i,j})$ of full rank.  
Let
$\KK[x_{i,j}]$ be the polynomial ring on the variables $x_{i,j}$. 
The ideal of the flag variety $\Flag_n$, denoted by $I_n$, is the kernel of the polynomial map
\vspace{-1mm}
\begin{eqnarray}\label{eq:map}
\varphi_n:\ \mathbb{K}[P_J: \varnothing\neq J\subsetneq [n]]\rightarrow\mathbb{K}[x_{i,j}]
\end{eqnarray}
sending each variable
$P_J$ to the determinant of the submatrix of $X$ with row indices $1,\ldots,|J|$ and column indices in $J$. We call the variables $P_J$ of the ring  Pl\"ucker variables and their images $\varphi_n(P_J)$ Pl\"ucker forms. We also call $I_n$ the Pl\"ucker ideal of the flag variety $\Flag_n$.

Given $v\leq w$ in $S_n$, we define the collection of subsets $T_w^v=\{J \subset[n] : v \le J \le w \}$ and its complement $S_w^v = \{J \subseteq [n] \} \backslash T_w^v$. The comparison of subsets and elements of $S_n$, along with the Bruhat order on $S_n$ is given in Section~\ref{sec:prelim_tableaux}. Then the associated ideal of the Richardson variety $X_w^v$ is
\begin{equation}\label{eq:restric}
I(X_w^v) = (I_n + \langle P_J: J\in S_w^v\rangle) \cap \mathbb K[P_J : J \in T_w^v]= (I(X_w)+I(X^v)) \cap \mathbb K[P_J : J \in T_w^v].
\end{equation}

\medskip

We now give an example of the subsets $S_w^v$ and $T_w^v$ 
(see \cite[\S 3.4]{kreiman2002richardson} for more details).

\begin{example}\label{def:sets}{\rm 
Let $n = 4$. Consider the permutations $v = (2314)$ and $w = (4231)$. The subsets of $[n]$ in $T^v_w$ of size one are given by those entries that lie between $v_1 = 2$ and $w_1 = 4$, which are $2, 3$ and $4$. The subsets of size two are those that lie between $\{v_1, v_2 \} = 23$ and $ \{w_1, w_2 \} = 24$ which are $23$ and $24$. The subsets of size three are those which lie between $\{v_1, v_2, v_3 \} = 123$ and $\{ w_1, w_2, w_3\} = 234$ which are all possible three-subsets. So we have: 
\[
{T}^{(2314)}_{(4231)} = \{2,3,4, 23,24,123,124,134,234 \} \text{ and } S^{(2314)}_{(4231)} = \{1, 12, 13, 14, 34\}.
\]
}\end{example}

\subsection{Gr\"obner degenerations of \texorpdfstring{$I_n$}{In}.}\label{subsec:flag}\hspace{0.3mm}
We first fix our notation throughout this section.
We fix the $(n-1) \times n$ matrix $M$ with entries:
\begin{eqnarray}\label{eq:matrix}
M_{i,j}=
      (i-1)(n-j+1)
\end{eqnarray}
For instance, when $n = 5$ we have the following matrix $4 \times 5$ matrix:
\[
M = 
\begin{bmatrix}
0  &  0  &  0  &  0  &  0  \\
5  &  4  &  3  &  2  &  1  \\
10 &  8  &  6  &  4  &  2  \\
15 &  12 &  9  &  6  &  3  
\end{bmatrix}
\]
Let $X=(x_{i,j})$ be an $ (n-1)\times n$ matrix of indeterminates. For each $k$-subset $J$ of $[n]$, the initial term of the Pl\"ucker form $\varphi_n(P_J) \in \mathbb{K}[x_{ij}]$ denoted by $\ini_{M}(P_J)$ is the sum of all terms in $\varphi_n(P_J)$ of the lowest weight, where the weight of a monomial $\bf m$ is the sum of entries in $M$ corresponding to the variables in $\bf m$. 
By \cite[Proposition~2.7]{Chary_Ollie_Fatemeh}, the initial term $\init_{M}(P_J)$ is the \textit{leading diagonal term} of the minor $\varphi_n(P_J)$ for each subset $J \subseteq [n]$. Explicitly, if $J = \{j_1 < \dots < j_k \}$ then we have
$
\init_{M}(P_J) = x_{1,j_1} x_{2,j_2} \dots x_{k,j_k}.
$
The weight of each variable $P_J$ is defined as the weight of each term of $\textrm{in}_{M}(P_J)$ with respect to $M$, and it is called {\em the weight induced by $M$}. We write ${\bf w}_M$ for the weight vector induced by $M$ on the Pl\"ucker variables.

Throughout this note, we will write $\init(I_n)$ for the initial ideal of $I_n$ with respect to ${\bf w}_M$. In the following theorem, we summarise some of the important properties of $\init(I_n)$ from \cite{OllieFatemeh2}. See also Theorem~14.16 in \cite{miller2004combinatorial} in which $\init(I_n)$ is realised as a Hibi ideal \cite{hibi1987distributive} associated to the poset whose underlying set consists of Pl\"ucker variables.

\begin{theorem}[Theorem 3.3 and Corollary 4.13 in \cite{OllieFatemeh2}]
\label{thm:Pure}    

The ideal $\init(I_n)$ is generated by quadratic binomials. Moreover, it is toric and it is equal to the kernel of the monomial map:
\begin{eqnarray}\label{eqn:monomialmapflags}
\phi_n \colon\  \mathbb{K}[P_J: \varnothing\neq J\subsetneq [n]]  \rightarrow \mathbb{K}[x_{ij}]  
\quad\text{with}\quad
 P_{J}   \mapsto {\rm in}_{M}(P_J).
\end{eqnarray}
\end{theorem}

\vspace{-2mm}
\subsection{Gr\"obner degenerations of \texorpdfstring{$I(X_w^v)$}{I(Xwv)}.}\label{subsec:X}\hspace{0.3mm}
For the Richardson variety $X_w^v$ we project the weight vector ${\bf w}_M$ induced by the matrix $M$ in \eqref{eq:matrix} to the coordinates corresponding to the variables in the polynomial ring 
$\KK[P_J: {J\in T_w^v}]$
and study its corresponding initial ideal $\init_{\bf w}(I(X_w^v))$ and its relation to the kernel of the monomial following map obtained by restricting the map $\phi_n$ from \eqref{eqn:monomialmapflags} to the polynomial ring $\KK[P_J: {J\in T_w^v}]$ as follows:
\begin{eqnarray}\label{eqn:restricted_map}
\phi_n|_w^v \colon\  \mathbb{K}[P_J: J\in T_w^v]  \rightarrow \mathbb{K}[x_{ij}]  
\quad\text{with}\quad
 P_{J}   \mapsto {\rm in}_{M}(P_J).
\end{eqnarray}
To simplify our notation we will omit the weight vector ${\bf w}_M$ from the initial ideals and write $\init(I(X_w^v))$. We also introduce the following notation to simplify the description of our ideals.

\begin{Notation}
\label{notation:restricted}
Let $G \subset \mathbb{K}[P_J: \varnothing\neq J\subsetneq [n]]$ a collection of polynomials and $T$ be a collection of subsets of $[n]$. We identify $T$ with the characteristic vector of $T^{\rm c}$, i.e.~$T_J = 1$ if $J \not\in T$ otherwise $T_J = 0$. For each $g \in G$ we write $g = \sum_{\alpha} c_\alpha {\bf P}^\alpha$ and define
\[
\hat g = \sum_{T \cdot\alpha = 0} c_\alpha {\bf P}^\alpha\quad\text{and}\quad 
G|_T= \{\hat g : g \in G \} \subseteq \mathbb{K}[P_J: J\in T].
\]
We call $\langle G|_T\rangle $ the {\it restriction} of the ideal $\langle G\rangle$ to $T$.
It is useful to think of $G|_T$ as the set obtained from $G$ by setting the variables $\{P_J: J \not\in T\}$ to zero. We say that the variable $P_J$ \textit{vanishes} in the ideal $\langle G|_T \rangle$ if $J \not\in T$. Similarly, we say that a polynomial $g \in \mathbb{K}[P_J:\ J\in T]$ vanishes in the restricted ideal $\langle G|_T \rangle$ if $g \in \langle P_J: J \not\in T \rangle$. The ideal $\langle G|_T \rangle$ can be computed in $\mathtt{Macaulay2}$~\cite{M2} as an elimination ideal using the following command
$$\mathtt{eliminate}(\langle G\rangle + \langle P_J: J \not\in T \rangle, \{P_J : J \in T\}).$$
\end{Notation}
\begin{lemma}
\label{lem:elim_ideal_gen_set}
With the notation above we have:
 \begin{itemize}
     \item[{\rm (i)}] $\langle G|_T \rangle = \langle G \cup \{P_J:\ J\not\in T\} \rangle \cap \mathbb{K}[P_J:\ J\in T]$.
\item[{\rm (ii)}] Let $v,w\in S_n$. Then the ideal $ \init(I_n)|_{T_w^v}$ is generated by quadratic binomials.
 \end{itemize}
\end{lemma}
\begin{proof}
Part (i) is \cite[Lemma~6.3]{OllieFatemeh3}. To prove (ii) we first note that by Theorem~\ref{thm:Pure}, there exists a set of quadratic binomials $G$ generating the ideal $\init(I_n)$. Hence, the assertion follows immediately from (i). In particular, we have that
\begin{equation}\label{eq:G|T}
    \init(I_n)|_{T_w^v}=\langle G|_{T_w^v}\rangle = \langle G \cup \{P_J : J \in S_w^v \} \rangle \cap \mathbb K[P_J : J \in T_w^v]
\end{equation}
which completes the proof of (ii).
\end{proof}

\subsection{Permutations, tableaux and their combinatorial properties.}\label{sec:prelim_tableaux}

In this section we introduce and prove some basic facts about \textit{semi-standard Young tableaux} and their \textit{defining chains}. We use these to study standard monomial bases for Richardson varieties in Section~\ref{sec:SMT}. 
We begin by recalling, from the beginning of  Section~\ref{sec:prim}, that $\le$ denotes the Bruhat order on $S_n$ and a comparison operator between subsets of $[n]$ and permutations.

A \textit{semi-standard Young tableau} $T$ is a sequence of subsets $I_1, \dots, I_d$ of $[n]$ such that $I_1 \le \dots \le I_d$. Each subset $I_j$ is called a column of $T$ and we will write this $T = [I_1 \dots I_d]$. For each pair of permutations $v, w$, we define $SSYT_d(v,w)$ to be the collection of all semi-standard Young tableau $T = [I_1 \dots I_d]$ such that $v \le I_k \le w$ for all $k \in [d]$ 

\begin{example}
It is often convenient to draw $T$ in a diagram, for example if $T = [I_1 I_2 I_3] = [125, 246, 35]$ then the corresponding diagram has columns $I_1, I_2, I_3$ and is drawn:
\[
T = 
\begin{tabular}{ccc}
    \hline
    \multicolumn{1}{|c|}{$1$} &
    \multicolumn{1}{ c|}{$2$} &
    \multicolumn{1}{ c|}{$3$} \\
    \hline 
    \multicolumn{1}{|c|}{$2$} &
    \multicolumn{1}{ c|}{$4$} &
    \multicolumn{1}{ c|}{$5$} \\
    \hline 
    \multicolumn{1}{|c|}{$5$} &
    \multicolumn{1}{ c|}{$6$} & \\
    \cline{1-2}
\end{tabular}
\]
Note that such diagrams are defined by: columns with weakly decreasing length, weakly increasing entries in each row and strictly increasing entries in each column.
\end{example}

\begin{definition}\label{def:defining_chain}
Let $T$ be a semi-standard Young tableau with columns $I_1, \dots, I_d$. Let $\textbf{u} = (\textbf{u}_1, \dots, \textbf{u}_d)$ be a sequence of permutations and write $\textbf{u}_k = (u_{k,1}, \dots, u_{k,n}) \in S_n$ for each $k \in [d]$. We say that $\textbf{u}$ is a \emph{defining chain} for $T$ if the permutations are monotonically increasing $\textbf{u}_1 \le \textbf{u}_2 \le \dots \le \textbf{u}_d$ with respect to the Bruhat order and for each $k \in [d]$ we have $I_k = \{u_{k,1}, \dots, u_{k, |I_k|} \}$.
\end{definition}

There is a natural partial order on the set of defining chains for a given semi-standard Young tableau $T$.
Let $\bm{\pi} = (\bm{\pi}_1, \dots, \bm{\pi}_t)$ and $\bm{\sigma} = (\bm{\sigma}_1, \dots, \bm{\sigma}_t)$ be defining chains for $T$. We say $\bm{\pi} \le \bm{\sigma}$ if $\bm{\pi}_k \le \bm{\sigma}_k$ for all $k \in [t]$. It turns out that there exists a unique minimum $w_-(T) = (w_1^-, \dots, w_d^-)$ and a unique maximum $w_+(T) = (w_1^+, \dots, w_d^+)$ defining chains for $T$. When the tableau is not clear from the context, we write $w_i^+(T)$ for $w_i^+$ and $w_i^-(T)$ for $w_i^-$. 

\medskip

The following notation is particularly useful for describing the permutations $w_1^+$ and $w_2^-$.

\begin{Notation}\label{notation:perms_from_partitions}
{\rm Let $\mathcal P = (P_1, P_2, \dots, P_k )$ be a $k$-partition of $[n]$ where $P_i$ are non-empty and disjoint subsets of $[n]$. Write $P_i = \{p_{i,1} < p_{i,2} < \dots < p_{i, |P_i|}\}$ for each $i$, and define the permutations:
\[
(P_1^\uparrow, P_2^\uparrow, \dots, P_k^\uparrow) = (p_{1,1}, p_{1,2}, \dots, p_{1, |P_1|}, p_{2,1}, \dots, p_{2, |P_2|}, p_{3,1}, \dots, p_{k, |P_k|}),
\]
\[
(P_1^\downarrow, P_2^\downarrow, \dots, P_k^\downarrow) = (p_{1,|P_1|}, p_{1,|P_1| - 1}, \dots, p_{1,1}, p_{2,|P_2|}, \dots, p_{2, 1}, p_{3,|P_3|}, \dots, p_{k, 1}).
\]
Note that the set $P_k$ is determined uniquely by $P_1, \dots, P_{k-1}$. So we write $(P_1^\uparrow, \dots, P_{k-1}^\uparrow, \uparrow)$ and $(P_1^\downarrow, \dots, P_{k-1}^\downarrow, \downarrow)$ for the above permutations respectively.
If any of the parts $P_i = \{p_{i,1}\}$ are singleton sets then we omit the arrow on that part from the notation. 
}
\end{Notation} 

We proceed by proving some basic properties of these permutations from partitions and their relationship to minimum and and maximum defining chains.

\begin{proposition}\label{prop:3_part_perm_up}\label{prop:3_part_perm_down}
Suppose $P_1, P_2, P_3$ is a $3$-partition of $[n]$. Let $v, w$ be permutations such that $P_1 \le w$ and $v \le P_1 \cup P_2$.
\begin{itemize}
    \item If $(P_1^\uparrow, P_2^\uparrow, P_3^\uparrow) \nleq w$ then $P_1 \cup P_2 \nleq w$.
    \item If $v \nleq (P_1^\downarrow, P_2^\downarrow, P_3^\downarrow)$ then $v \nleq P_1$.
\end{itemize}
\end{proposition}

\begin{proof}
For the permutation $w$, the proof follows from the fact that if $P_1 \le w$ and $P_1 \cup P_2 \le w$ then $(P_1^\uparrow, P_2^\uparrow, P_3^\uparrow) \le w$. For the permutation $v$, the proof follows from the fact that if $v \le P_1$ and $v \le P_1 \cup P_2$ then $v \le (P_1^\downarrow, P_2^\downarrow, P_3^\downarrow)$.
\end{proof}

\begin{proposition}\label{prop:w_three_parts}
Let $T = [IJ]$ be a semi-standard Young tableau with two columns $I$ and $J$. We have that $w_2^- = (J^\uparrow, I_-^\uparrow, \uparrow)$ and $w_1^+ = (I \backslash I_+^\downarrow, I_+^\downarrow, \downarrow)$ for some subsets $I_+, I_- \subseteq I$.
\end{proposition}

\begin{proof}
For any permutation $w = (w_1, \dots, w_n) \in S_n$ and $k \le n$, we write $w([k]) = \{ w_1, \dots, w_k\}$. We define $s = |J|$ and $t = |I|$ for the size of the columns of $T$. Note that $w_1^- = (I^\uparrow, \uparrow)$ is the smallest permutation such that $w_1^-([t]) = I$, i.e for all permutations $v$ with $v([t]) = I$ we have $w_1^- \le v$. Similarly $w_2^+ = (J^\downarrow, \downarrow)$ is the greatest permutation such that $w_2^+([s]) = J$.
It follows from the definition that $w_2^-$ is the smallest permutation such that $w_1^- \le w_2^-$ and $w_2^-([s]) = J$. It easily follows that $w_2^-$ has the desired form. Similarly, by definition, $w_1^+$ is the greatest permutation such that $w_1^+ \le w_2^+$ and $w_1^+([t]) = I$. And so $w_1^+$ has the desired form.
\end{proof}

\subsection{Standard monomials.}\label{sec:standard_monomial}
The Pl\"ucker algebra 
of the flag variety is given by  $\mathbb{K}[P_J]_{J\subset [n]}/I_{n}$, where $I_n$ is the Pl\"ucker ideal. 
A monomial $P = P_{J_1}\dots P_{J_d}$ is called \textit{standard for $\Flag_n$} if $J_1 \le \dots \le J_d$. A \textit{standard monomial basis} of a Richardson variety is a collection of standard monomials which form a basis for the corresponding Pl\"ucker algebra.
We refer to \cite[\S 2.2]{seshadri2016introduction} for more details.
To simplify our notation, we identify the semi-standard Young tableau $T = [J_1, \dots, J_d]$ with the monomial $P_{J_1}\dots P_{J_d} \in \KK[P_J]$, following the notation of Section~\ref{sec:prelim_tableaux}. The standard monomials for Richardson varieties can be determined by minimum and maximum defining chains, as follows (see e.g.~\cite[Theorem~34]{lakshmibai2003richardson}).

\begin{theorem}
\label{prop:Kim}
Let $v \le w$ be permutations. The collection of semi-standard Young tableaux $T$ such that $w^-_{|T|} \le w$ and $v \le w^+_{1}$ forms a monomial basis for $X_w^v$, where $w_-(T) = (w^-_1, \dots, w^-_{|T|})$ and $w_+(T) = (w^+_1, \dots, w^+_{|T|})$ are the unique minimum and maximum defining chains for $T$, respectively.
\end{theorem}

\section{Standard monomials 
}\label{sec:SMT}

The description of the standard monomials for Richardson varieties in Theorem~\ref{prop:Kim} can be combinatorially difficult to determine. The goal of this section is to prove
Theorem~\ref{thm:ssyt_fl_diag} which is our main result and gives a very simple description of the standard monomials for the Richardson varieties $X_w^v$ with $(v,w) \in \MT_n$. 
We note that, in contrast to the Grassmannian, the monomials associated to the tableaux $SSYT_d(v,w)$ may not constitute a monomial basis for the Richardson variety $X_w^v$ inside the flag variety, see Example~\ref{example1}. However, if $(v, w)\in \MT_n$ then Theorem~\ref{thm:ssyt_fl_diag} shows that the semi-standard Young tableaux $SSYT_d(v,w)$ do in fact form a standard monomial basis for the Richardson variety $X_w^v$.

\medskip
We will now introduce the set of pairs of permutations $\MT_n$ inductively. In order to define this set we require the following. For any permutation $w  = (w_1, \dots, w_n) \in S_n$ with $w_t = n$ for some $t \in [n]$, we define the \textit{induced permutation} $\underline w = (w_1, \dots, w_{t-1}, w_{t+1}, \dots, w_n) \in S_{n-1}$.

\begin{definition}[Compatible pairs]\label{def:compatible}
Let $v,w\in S_n$ with $n = v_t = w_{t'}$ and $n-1 = v_s = w_{s'}$. We say that $(v, w)$ is a \textit{compatible} pair if either (i) $t = t'$ or (ii) $t' < t$ and in this case the following conditions hold:
\[
s'\le t, \quad t'\le s, \quad n=w_{t'}>w_{t'+1}>\cdots>w_t,\quad\text{and}\quad n=v_{t}>v_{t-1}>\cdots>v_{t'}.
\]
We define the set of pairs of permutations $\MT_n \subseteq S_n \times S_n$ inductively by
\begin{equation}\label{eq:def_TN}
\MT_1 = \{ (id, id)\}, \ 
\MT_{n+1} = \{(v,w) \in S_{n+1} \times S_{n+1} : (\underline v, \underline w) \in \MT_n \textrm{ and } (v,w) \textrm{ compatible}\}.
\end{equation}
\end{definition}

\begin{example} 
{\rm Consider $\big((1,3,2),(2,3,1)\big)$ in $\MT_{3}$. 
We find all $(v,w)$ in $\MT_{4}$ with $\underline{v} = (1,3,2)$ and $\underline{w} = (2,3,1)$. Firstly, if we have $4 = v_t = w_t$ for some $t$ then we have the following pairs:
   \[
   \big((1,3,2,\textbf{4}),(2,3,1,\textbf{4})\big),\quad 
    \big((1,3,\textbf{4},2),(2,3,\textbf{4},1)\big),\]
    \[
    \big((1,\textbf{4},3,2),(2,\textbf{4},3,1)\big),\quad 
    \big((\textbf{4},1,3,2),(\textbf{4},2,3,1)\big).
    \]
Secondly, for a compatible pair $(v,w)$ such that $v_t = 4$, $w_{t'} = 4$, assume that $t \neq t'$. Since $v \le w$ we have that $t > t'$. So we get the pair:
 $((1,3,\textbf{4},2),(2,\textbf{4},3,1))$.
}
\end{example}

\begin{example}\label{example1} Consider  $\Flag_3$ and $w=(3,1,2)\in S_3$. Note that $P_{23}P_1=P_{13}P_2-P_{12}P_3$.
Consider the tableaux below:
\[
T_1 = 
\begin{tabular}{cc}
    \hline
    \multicolumn{1}{|c|}{1} & \multicolumn{1}{c|}{3} \\ \hline
    \multicolumn{1}{|c|}{2} &   \\ \cline{1-1}
\end{tabular}
\quad \text{and} \quad T_2=\begin{tabular}{cc}
    \hline
    \multicolumn{1}{|c|}{1} & \multicolumn{1}{c|}{2} \\ \hline
    \multicolumn{1}{|c|}{3} &   \\ \cline{1-1}
    \end{tabular}.
\]
So we have $P_{23}P_{1}= T_2 - T_1$. We have that 
$P_{23}$ vanishes on $X_w$, since $\{2,3\}\nleq\{3,1\}$ and by the defining ideal of the Schubert variety $X_w$. 
And so $T_1$ and $T_2$ are equal in the coordinate ring of the Schubert variety $X_w$, in particular $T_1$ and $T_2$ are linearly dependent.
Therefore, the set $SSYT_1(id, w)$ is not a monomial basis for the Schubert variety $X_w = X_w^{id}$.
\end{example}

\noindent

\subsection{Block structure on compatible permutations.}\label{sec:Tn_block_structure}
To prove Theorem~\ref{thm:ssyt_fl_diag}, we need to introduce a \textit{block} structure on the pairs $(v,w) \in \MT_n$. 

\begin{definition}
Let $(v,w) \in \MT_n$ and $v_d = w_e = n$ for some $e, d \in [n]$. A \textit{block} of $(v,w)$ is a pair of consecutive subsets of $v$ and $w$ on the same indices which we write as $(v,w)_i^j = (\{v_i, v_{i+1}, \dots, v_j\}, \{ w_i, w_{i+1}, \dots, w_j\})$ for some $i \le j$ and satisfies one of the following criteria. Note that the persistence and expansion criteria are defined inductively on $n$.
\begin{itemize}
    \item (Creation) If $i = j = e = d$ then $(v,w)_i^j$ is a block.
    
    \item (Persistence) Assume that $n \notin \{v_i, \dots, v_j \}$ and $n \notin \{w_i, \dots, w_j \}$. If $(v,w)_i^j$ is a block for $(\underline v, \underline w)$ then $(v,w)_i^j$ is a block for $(v, w)$.
    
    \item (Expansion) Assume that $n \in \{v_i, \dots, v_j \}$ and $n \in \{w_i, \dots, w_j \}$. In addition, assume that $i < d$ and  $e < j$. If 
    \[
    (\underline v, \underline w)_i^{j-1} = 
    (\{v_i, \dots, v_{d-1},v_{d+1} \dots, v_j\},
    \{w_i, \dots, w_{e-1}, w_{e+1}, \dots, w_j \})
    \]
    is a block for $(\underline v, \underline w)$ then $(v,w)_i^j$ is a block for $(v,w)$.
\end{itemize}
The \textit{size} of a block $(v,w)_i^j$ is equal to $j - i + 1$.
\end{definition}

\begin{example}{\rm Here we give two examples illustrating properties of blocks.
\begin{itemize}
    \item Let $v = (3,5,6,4,1,2)$ and $w = (4,6,5,3,2,1)$ then there are three distinct blocks
    \[
    (v,w)_5^6 = (\{1,2\}, \{2,1\}), \quad 
    (v,w)_1^4 = (\{3,5,6,4\}, \{4,6,5,3\}), \quad 
    (v,w)_2^3 = (\{5,6\}, \{6,5\}).
    \]
    We see that blocks are either disjoint: such as $(v,w)_5^6$ and $(v,w)_1^4$, or subsets of one another: such as $(v,w)_1^4$ and $(v,w)_2^3$.
    
    \item For $v = (1,2,4,5,3)$ and $w = (2,4,5,3,1)$ the only block is $(v, w)_1^5$. Note that, $1 < 2 < 4 < 5$ is an increasing sequence in $v$ and $5 > 3 > 1$ is a decreasing sequence in $w$.
\end{itemize}
}\end{example}

\begin{definition}
We say two distinct blocks $(v,w)_{i}^j$ and $(v,w)_{k}^\ell$ are \textit{crossing} if $i<k<j<\ell$ or $k < i < \ell < j$. Otherwise, they are called \textit{non-crossing} and must be either disjoint: i.e.~$j < k$ or $\ell < i$, or contained in one another: i.e.~$i < k < \ell < j$ or $k < i < j < \ell$. 
\end{definition}

\begin{proposition}\label{prop:unified}
Let $(v,w) \in \MT_n$, then the following hold:
\begin{itemize}
    \item[{\rm(a)}] For any block $(v,w)_{i}^j$  we have $\{v_i, \dots, v_j \} = \{w_i, \dots, w_j \}$ and $v_i = w_j = \min\{v_i, \dots, v_j \}$.
    
    \item[{\rm(b)}] Any pair of distinct blocks are non-crossing.
\end{itemize}
\end{proposition}

\begin{proof}
For (a) we proceed by double induction: first on $n$ and then on the size of the block which is $j - i + 1$. If the block has size $1$, i.e.~$i = j$, then the result holds trivially. So let us assume that $j - i \ge 1$. If $(v,w)_i^j$ is a block for $(\underline v, \underline w)$ then the results follows by induction on $n$. Otherwise if $(v,w)_i^j$ is not a block of $(\underline v, \underline w)$ then $(v,w)_i^j$ must occur by the expansion criterion and so by induction the result holds.

Part (b) follows easily by induction $n$, noting that if two distinct blocks are crossing then each has size at least two and must arise, either by the persistence or expansion criteria, from a pair of crossing blocks of $(\underline v, \underline w)$.
\end{proof}

\begin{definition}
Let $(v,w) \in \MT_n$. By Proposition~\ref{prop:unified}(b) we have that the blocks containing $v_d = w_e = n$ are totally ordered by inclusion. The smallest such block is called the \textit{maximum block of $(v,w)$}.
\end{definition}

\begin{proposition}\label{prop:block_max_order}
Let $(v,w) \in \MT_n$ and $(v,w)_{i}^j$ 
be the maximum block of $(v,w)$. Write $v_d = w_e = n$ for some $d, e \in [n]$. Then we have that $v_i < v_{i+1} < \dots < v_d$ and $w_e > w_{e+1} > \dots > w_j$.
\end{proposition}

\begin{proof}
Throughout the proof we write $w_{e'} = v_{d'} = n-1$ for some $e', d' \in [n]$. We proceed by double induction, first by induction on $n$ and then by induction on the size of the maximum block. For any $n$, if the maximum block has size one then the result trivially holds. For the induction step, assume the maximum block has size at least two. It immediately follows that $e < d$, so by compatibility of $v$ and $w$, we have: $v_e < \dots < v_d$, $w_e > \dots > w_d$, $d' \ge e$ and $e' \le d$.

Let us write $(\underline v, \underline w)_k^{\ell - 1}$ for the maximum block of $(\underline v, \underline w)$ for some $k,\ell \in [n]$, which is the smallest block containing $n-1$. Since $d' \ge e$ and $e' \le d$, by the expansion criterion, we have that $(v,w)_k^\ell$ is a block containing $n$. Since $(v,w)_i^j$ is the maximum block of $(v,w)$ we have that $(v,w)_k^\ell$ is contained in $(v,w)_i^j$, i.e.~$k \le i$ and $j \le \ell$. Since $n$ is contained in the maximum block of $(v,w)$, we must have that the maximum block is obtained from the block $(\underline v, \underline w)_{i}^{j-1}$ in $(\underline v, \underline w)$ by the expansion criterion. Therefore, $(\underline v, \underline w)_{i}^{j-1}$ is a block contained in the maximum block of $(\underline v, \underline w)$. However, it can be shown directly from the definition that the maximum block does not properly contain any other blocks. Therefore, $(\underline v, \underline w)_{i}^{j-1} = (\underline v, \underline w)_{k}^{\ell-1}$ is the maximum block and so by induction we have $v_i < \dots < v_{d'}$ and $w_{e'} > \dots > w_j$. By compatibility we have $v_e < \dots < v_d$ and $w_e > \dots > w_d$. Since $d' \ge e$ and $e' \le d$, it follows that $v_i < \dots < v_d$ and $w_e > \dots > w_j$.
\end{proof}

\begin{proposition}\label{prop:max_block_elements}
Let $(v,w) \in \MT_n$ and $(v,w)_i^j$ be the maximum block of $(v,w)$. Then we have $\{v_i, v_{i+1}, \dots, v_j \} = \{w_i, w_{i+1}, \dots, w_j \} = \{v_i, v_i + 1, \dots, n \} = \{w_j, w_j + 1, \dots, n \}$.
\end{proposition}

\begin{proof}
By Proposition~\ref{prop:unified}(a), it suffices to show that $\{v_i, v_{i+1}, \dots, v_j \} = \{v_i, v_i + 1, \dots, n \}$. We proceed by induction on the size of the maximum block $j - i + 1$. If the maximum block has size $1$, then $(v,w)_i^j = (\{n \}, \{n \})$ and the result holds trivially. 

Assume that the maximum block has size $j - i + 1 \ge 2$ and write $v_d = w_e = n$ and $w_{e'} = v_{d'} = n-1$ for some $d,d',e,e' \in [n]$. Since $(v,w) \in \MT_n$, we have that $v \le w$ and so $e \le d$. If $e = d$, then $(v,w)_d^d$ is a block that is properly contained in the maximum block $(v,w)_i^j$, a contradiction. So we must have $e < d$. In particular, $e < j$ and $i < d$. By the expansion criterion, we have that $(\underline v, \underline w)_i^{j-1}$ is a block for $(\underline v, \underline w)$. 

We proceed by showing that $(\underline v, \underline w)_i^{j-1}$ is the maximum block for $(\underline v, \underline w)$. To do this, we first show that $n-1 \in \{v_i, \dots, v_j \}$. Assume by contradiction that $n-1 \notin \{v_i, \dots, v_j \} = \{w_i, \dots, w_j \}$. Since $(v,w)$ are compatible, we have that $d' \ge e$ and $e' \le d$. Therefore $d' \ge j+1$ and $e' \le i - 1$. By the persistence criterion, we have that $\left(\underline{\underline{v} }, \underline{\underline{w} } \right)_i^{j-1}$ is a block for $\left(\underline{\underline{v} }, \underline{\underline{w} } \right)$. We write $\left(\underline{\underline{v} }, \underline{\underline{w} } \right)_i^{j-1} = (\{\tilde v_i, \dots, \tilde v_{j-1}\}, \{\tilde w_i, \dots, \tilde w_{j-1}\})$. Since $i < d \le j < d'$, we have that $\{\tilde v_i, \dots, \tilde v_{j-1}\} =  \{v_i, \dots,  v_j\} \backslash \{n\}$. Since $e' < i \le e < j$, we have that $\{\tilde w_i, \dots, \tilde w_{j-1}\} = \{w_{i+1}, \dots, w_{j+1}\} \backslash \{n\}$. By Proposition~\ref{prop:unified}(a) applied to $(v,w)_i^j$ we have that $\{v_i, \dots, v_j\} = \{w_i, \dots,  w_j\}$. Since $(v, w) \in \MT_n$, we have that $\left(\underline{\underline{v} }, \underline{\underline{w} } \right) \in \MT_{n-2}$. So, by Proposition~\ref{prop:unified}(a) applied to $\left(\underline{\underline{v} }, \underline{\underline{w} } \right)_i^{j-1}$, we have that $\{\tilde v_i, \dots, \tilde v_{j-1}\} = \{\tilde w_i, \dots, \tilde w_{j-1}\}$. However, by the above we have $w_i \in \{\tilde v_i, \dots, \tilde v_{j-1}\}$ but $w_i \notin  \{\tilde w_i, \dots, \tilde w_{j-1}\}$, a contradiction.

Next, we show that $(\underline v, \underline w)_i^{j-1}$ does not properly contain another block. Assume by contradiction that $(\underline v, \underline w)_{i'}^{j'-1}$ is a block containing $n-1$ and is properly contained in $(\underline v, \underline w)_i^{j-1}$. By compatibility, we have that $d' \ge e$ and $v_e < v_{e+1} < \dots < v_{d} = n$. It follows that $d' \ge d-1$, and so $j' \ge d$. Similarly, by compatibility, we have $e' \le d$ and $n = w_e > w_{e-1} > \dots > w_d$. If follows that $e' \le e + 1$, and so $i' \le e$. By the expansion criterion, we have that $(v,w)_{i'}^{j'}$ is a block for $(v,w)$ that is strictly contained in the maximum block, a contradiction.

So we have shown that $(\underline v, \underline w)_i^{j-1}$ is the maximum block of $(\underline v, \underline w)$. The result follows immediately by induction. \qed

\end{proof}
\color{black}

\subsection{Proof of main result.}\label{sec:ssyt_proof_main}

In this section we prove Theorem~\ref{thm:ssyt_fl_diag}. To do this we require the following construction. Let $(v,w) \in \MT_n$ be a pair of permutations and write $w_e = v_d = n$ for some integers $e$ and $d$. Assume that $e < d$. Define the permutations 
\[
w' = (w_1, \dots, w_{e-1}, w_{e+1}, w_e, w_{e+2}, \dots, w_n)\text{ and }v' = (v_1, \dots, v_{d-2}, v_d, v_{d-1}, v_{d+1}, \dots, v_n).
\]

\begin{lemma}\label{lem:inductive_vw}
For a given pair $(v,w) \in \MT_n$,  we have $(v', w), (v, w') \in \MT_n$.
\end{lemma}

\begin{proof}
We note that $\underline v = \underline v'$ and $\underline w = \underline w'$ so it suffices to check that $(v',w)$ and $(v,w')$ are compatible pairs. Define $e', d' \in [n]$ such that $w_{e'} = v_{d'} = n-1$. Since $(v,w)$ are compatible, we have $d' \ge d - 1$ and $e' \le e + 1$. If $e = d-1$ then we have that the position of $n$ in each of the pairs $(v,w')$ and $(v',w)$ is the same and so each is a compatible pair. If $e < d - 1$ then we have that $e+1 \le d-1 \le d'$ and so $(v,w')$ is a compatible pair, and similarly we have $d-1 \ge e+1 \ge e'$ and so $(v', w)$ is a compatible pair.
\end{proof}

This construction is useful for our inductive argument in the proof of Theorem~\ref{thm:ssyt_fl_diag}. In particular, we will induct on the dimension of the Richardson variety $X_w^v$ which can be read combinatorially from the inversion numbers $N(v)$ and $N(w)$ of the permutations $v$ and $w$ respectively.  Note that $N(v') = N(v) + 1$ and $N(w') = N(w) - 1$. Therefore, $\dim(X_{w'}^v) = \dim(X_w^{v'}) = \dim(X_w^v) - 1$.

\begin{theorem}\label{thm:ssyt_fl_diag}
Let $d \ge 1$ be a natural number and $(v, w) \in \MT_n$ be a pair of permutations. Then the number of standard monomials for $X_w^v$ in degree $d$ is equal to $|SSYT_d(v,w)|$.
\end{theorem}
\begin{proof}
Let $(v,w) \in \MT_n$ be a pair of permutations. We note that for $d = 1$ the result holds immediately. Since the ideal of the Richardson variety $X_w^v$ is generated by homogeneous quadrics, it suffices to show that any semi-standard Young tableau $T$ with two columns $I, J$ such that $v \le I, J \le w$ is standard for the Richardson variety $X_w^v$. We do this by showing that $v \le w_1^+ $ and $w_2^- \le w$.

We note that if $s = t$ then the defining permutations have a particularly simple description, i.e.~$w_1^+ = (I^\downarrow, \downarrow) \ge v$ and $w_2^- = (J^\uparrow, \uparrow) \le w$ and the result immediately follows. For the remaining cases, the proof proceeds by induction on $n$. If $n = 1$ then the result is trivial. For each $n > 1$ we proceed by induction on the dimension of the Richardson variety $X_w^v$. If the dimension is zero then we have that $v = w$ and the result is trivial. Let us assume that $v < w$.  We write
\[
    I = \{i_1 < \dots < i_t \}, \quad
    J = \{j_1 < \dots < j_s \}, \quad
    w = (w_1, \dots, w_n), \quad
    v = (v_1, \dots, v_n).
\]
Let $e, d \in [n]$ be integers such that $w_e = v_d = n$.

We proceed by taking cases on $e$ and $d$.

\medskip

\textbf{Case 1.} Assume that both $e$ and $d$ lie in one of the sets: $\{ 1, \dots, s\}, \{s+1, \dots, t \}$ or $\{t+1,\dots, n \}$. We define the tableau $T'$ with columns $I', J'$ as follows.
\begin{itemize}
    \item If $e,d \in \{ 1, \dots, s\}$ then define $I' = I \backslash n$ and $J' = J \backslash n$.
    \item If $e,d \in \{s+1, \dots, t \}$ then define $I' = I \backslash n$ and $J' = J$.
    \item If $e,d \in \{t+1,\dots, n \}$ then define $I' = I$ and $J' = J$.
\end{itemize}
By construction we have $\underline v \le I', J' \le \underline w$. Since $(\underline v, \underline w) \in \MT_n$, by induction on $n$ we have $T'$ is standard for $X_{\underline w}^{\underline v}$. We write $\underline w_- = (\underline w_1^-, \underline w_2^-)$ and $\underline w_+ = (\underline w_1^+, \underline w_2^+)$ for the minimum and maximum defining sequences for $T'$ in $S_{n-1}$. Since $T'$ is standard for $X_{\underline w}^{\underline v}$ we have $\underline w_2^- \le \underline w $ and $\underline v \le \underline w_1^+$. 

By Proposition~\ref{prop:w_three_parts} we have that $\underline w_2^- = (J'^\uparrow, I_-^\uparrow, \uparrow)$ and $\underline w_1^+ = ((I' \backslash I_+)^\downarrow, I_+^\downarrow, \downarrow)$ in $S_{n-1}$ for some subsets $I_+, I_- \subseteq I'$. It follows by the same proposition that:
\begin{itemize}
    \item If $e,d \in \{ 1, \dots, s\}$ then 
    $w_2^- = (J' \cup \{n\}^\uparrow, I_-^\uparrow, \uparrow)$ and 
    $w_1^+ = ((I' \backslash I_+) \cup \{n\}^\downarrow, I_+^\downarrow, \downarrow)$,
    \item If $e,d \in \{s+1, \dots, t \}$ then 
    $w_2^- = (J^\uparrow, I_- \cup \{n\}^\uparrow, \uparrow)$ and 
    $w_1^+ = ((I' \backslash I_+)^\downarrow, I_+\cup \{n\}^\downarrow, \downarrow)$,
    \item If $e,d \in \{t+1,\dots, n \}$ then $w_2^- = (J^\uparrow, I_-^\uparrow, \uparrow)$ and $w_1^+ = ((I' \backslash I_+)^\downarrow, I_+^\downarrow, \downarrow)$.
\end{itemize}
Since $\underline w_2^- \le \underline w$ and $\underline v \le \underline w_1^+$, it follows that $w_2^- \le w$ and $v \le w_1^+$.

\medskip

For the remaining cases note that we have $e < d$. So we recall the permutations 
\[
    w' = (w_1, \dots, w_{e-1}, w_{e+1}, w_e, w_{e+2}, \dots, w_n) 
    \textrm{ and }
    v' = (v_1, \dots, v_{d-2}, v_d, v_{d-1}, v_{d+1}, \dots, v_n).
\]
Note that by Lemma~\ref{lem:inductive_vw} we have that $(v',w), (v,w') \in \MT_n$.

\medskip

\textbf{Case 2.} Assume $s \le e \le t$ and $t+1 \le d$. If $d > t+1$ then it follows that $v' \le I, J \le w$ and so by induction we have $w_2^- \le w$ and $v \le v' \le w_1^+$. Therefore, we may assume that $d = t+1$. Similarly, if $e < t$ then it follows that $v \le I, J \le w'$ and so by induction we have $w_2^- \le w' \le w$ and $v \le w_1^+$. So we may assume that $e = t$.

\medskip

\noindent\textbf{Claim.} \textit{Either $v' \le I$ or $I \le w'$.}\label{claim:compatible_v_w}

To prove the claim, we proceed by taking cases on $I$, either $n \in I$ or $n \notin I$.

\textbf{Case i.} Assume $n \in I$. Since $w_t = n$, it follows that $I \nleq w'$ and so we will show that $v' \le I$. 
Let $(v,w)_i^j$ be the maximum block. Since $e = t$ and $d = t+1$ we have that $i \le t$ and $j \ge t+1$. By Proposition~\ref{prop:block_max_order} we have that $v_i < \dots < v_t < v_{t+1} = n$. By Proposition~\ref{prop:max_block_elements}, \color{black} the elements appearing in the maximum block are precisely $\{v_i, v_{i+1}, \dots, v_j \} = \{v_i, v_i +1, \dots, n \}$. Therefore, for all $k < i$ we have $v_k < v_i$ and so $v_t = \max\{v_1, \dots, v_t \}$. Since $n \in I$, it follows easily that $v' \le I$.

\medskip

\textbf{Case ii.} Assume $n \notin I$. Since $v_{t+1} = n$, it follows that $v' \nleq I$ and so we will show that $I \le w'$. Similarly to the above case, we consider the maximum block $(v,w)_i^j$. By Proposition~\ref{prop:block_max_order} we have that $n = w_t > \dots > v_j$. By Proposition~\ref{prop:max_block_elements}, \color{black} the elements appearing in the maximum block are precisely $\{w_i, w_{i+1}, \dots, w_j \} = \{w_j, w_j +1, \dots, n \}$. Therefore, for all $k > j$ we have that $w_k < w_j$ and so $w_{t+1} = \max\{w_{t+1}, \dots, w_n \}$. Since $n \notin I$, it follows easily that $I \le w'$. And so we proved the claim.

If $v' \le I$ then we have $v' \le I, J \le w$ and so by induction we have $T$ is standard for $X_{w}^{v'}$. Hence, $v \le v' \le w_1^+$ and $w_2^- \le w$. Therefore, $T$ is standard for $X_w^v$. On the other hand, if $I \le w'$ then we have $v \le I, J \le w'$ and so by induction we have $T$ is standard for $X_{w'}^v$. Hence, $v \le w_1^+$ and $w_2^- \le w' \le w$. Therefore, $T$ is standard for $X_w^v$.

\medskip

\textbf{Case 3.} Assume $e \le s$ and $s+1 \le d \le t$. This case is identical to Case 2 by considering the subset $J$ instead of $I$.

\medskip

\textbf{Case 4.} Assume $e \le s$ and $t+1 \le d$. We can assume as, similarly to Case 2, that $d = t+1$ and $e = s$. Note that $s < t$ and so $v \le v' \le w' \le w$. If either $v' \le I $ or $J \le w'$ then we can deduce the result by induction, so we will assume that both conditions do not hold. Equivalently, we will assume that $n \in J$ and $n \notin I$.
Let $(v,w)_i^j$ be the maximum block of $(v,w)$. By Proposition~\ref{prop:block_max_order} we have $v_i < \dots < v_d$ and $w_e > \dots > w_j$ and by Proposition~\ref{prop:unified}(a) for all $k < i$ and $k > j$ we have $v_k < v_i$ and $w_k < w_j$. Let us consider the size of the maximum block. Since $n \in \{w_1, \dots, w_s \}$ we have that $i \le s$. Since $n \in \{v_{t+1}, \dots, v_n \}$ we have that $j \ge t+1$. In particular, the indices of the maximum block span $\{s, s+1, \dots, t+1 \}$.

We now show that $v \le w_1^+$ and $w_2^- \le w$. Recall by Proposition~\ref{prop:w_three_parts} that $w_1^+ = (I \backslash \tilde I^\downarrow, \tilde I^\downarrow, \downarrow)$ and $w_2^- = (J^\uparrow,  I'^\uparrow, \uparrow)$ for some subsets $\tilde I$ and $I'$ of $I$. Since $v_s < \dots < v_t$ are the largest elements of $\{v_1, \dots, v_t \}$ and $v \le I$, it follows that $v \le w_1^+$.
Similarly, since $w_s > \dots > w_{t+1}$ are the largest elements in $\{w_1, \dots, w_t\}$ and $I, J \le w$, it follows that $w_2^- \le w$.
\end{proof}

\begin{remark}\label{rmk:ssyt_corollaries_schubert_op_schubert}
Schubert and opposite Schubert varieties are special examples of Richardson varieties. For these cases Theorem~\ref{thm:ssyt_fl_diag} has a particularly simple combinatorial description. A Schubert variety is a Richardson varieties $X_{w}^v$ such that $v = id$. It is easy to show that $(id, w) \in \MT_n$ if and only if $w$ is a 
$312$-avoiding
permutation. On the other hand, opposite Schubert varieties are Richardson varieties $X_w^v$ such that $w = w_0 = (n,n-1, \dots, 1)$. In this case $(v, w_0) \in \MT_n$ if and only if $v$ is 
$213$-avoiding.
\end{remark}

\section{Monomial-free ideals 
}\label{sec:mon_free_ideals}

We recall the definition of the ideal $\init(I_n)|_{T_w^v}$ from Sections~\ref{subsec:flag} and \ref{subsec:X} and the collection of pairs of permutations $\MT_n \subseteq S_n \times S_n$ from 
\eqref{eq:def_TN}. Our main result in this section is the following which gives a complete characterisation for monomial-free ideals of form $\init(I_n)|_{T_w^v}$.

\begin{theorem}
\label{thm:toric_fl_rich_diag}
The ideal $\init(I_n)|_{T_w^v}$ is monomial-free if and only if $(v,w) \in \MT_n$.
\end{theorem}

\begin{proof}
The proof follows directly from Lemmas~\ref{non toric}, \ref{comp} and \ref{lem:compatible_is_monomial_free}. In particular, Lemmas~\ref{non toric} and \ref{comp} show that if $\init(I_n)|_{T_w^v}$ is monomial-free then $\init(I_{n-1})|_{T_{\underline w}^{\underline v}}$ is monomial-free and $(v,w)$ is a compatible pair. And so, by induction on $n$, we have that if $\init(I_n)|_{T_w^v}$ is monomial-free then $(v,w) \in \MT_n$. On the other hand, Lemma~\ref{lem:compatible_is_monomial_free} shows that if $(v,w) \in \MT_{n}$ then $\init(I_n)|_{T_w^v}$ is monomial-free.
\end{proof}

\begin{example}\label{ex:non-compatible}
Let $v = (1,3,2)$ and $w = (3,1,2)$, which are non-compatible permutations. Note that the ideal $\init(I_2)|_{T_{\underline w}^{\underline v}} = 0$, in particular it is monomial-free. We also have $\init(I_3) = \langle P_{13}P_2 - P_{23}P_1 \rangle$, Since $T_w^v = \{1,2,3,13\}$, it follows that the ideal $\init(I_3)|_{T_w^v} = \langle P_{13} P_2\rangle$ contains a monomial. This monomial arises from the generator of $\init(I_3)$, where $P_{23}P_{1}$ vanishes in $\init(I_3)|_{T_w^v}$.
\end{example}

We now proceed to prove the lemmas used in the proof of Theorem~\ref{thm:toric_fl_rich_diag}.
We will first show, in Lemma~\ref{non toric}, that
\[
\init(I_{n+1})|_{T_w^v} \text{ monomial-free }\implies \init(I_n)|_{T^{\underline v}_{\underline w}}\text{ monomial-free.}
\]
However, the converse does not hold, see Example~\ref{ex:non-compatible}. We will show that compatibility is an essential ingredient in showing that $\init(I_{n+1})|_{T_w^v}$ is monomial-free. 
In Lemma~\ref{lem:compatible_is_monomial_free}, we will show that the converse holds with the added assumption of compatibility
\[
\init(I_n)|_{T^{\underline v}_{\underline w}}\text{ monomial-free and }(v,w) \text{ compatible }\implies \init(I_{n+1})|_{T_w^v} \text{ monomial-free.}
\]
\begin{lemma}\label{non toric}

Let $v, w \in S_{n+1}$ with $\underline v \le \underline w$. If $\init(I_{n+1})|_{T_w^v}$ is monomial-free 
then $\init(I_n)|_{T^{\underline v}_{\underline w}}$ is monomial-free.
\end{lemma}

\begin{proof} 
Suppose that $F_n|^{\underline v}_{\underline w}$ contains a 
monomial $P_IP_J$ 
which arises from the binomial $P_IP_J-P_{I'}P_{J'}$ in $\init_{W_D}(F_n)$. 
We construct a monomial in $F_{n+1}|_w^v$ as follows.  
Assume that  $|I|=|I'|\geq |J|=|J'|$.
Let $1\leq t'\leq t\leq n+1$ with $v_t=w_{t'}=n+1$.  We take cases on $t$,  $t'$, $|I|$, and $|J|$. 

\medskip

{\bf Case 1.} Assume that $t'>|I|$. Then $P_IP_J-P_{I'}P_{J'}$ is a binomial in  $\init_{W_D}(F_{n+1})$. It is clear that $P_{I}P_{J}$ does not vanish in $F_{n+1}|_w^v$ and $P_{I'}P_{J'}$ vanishes in $F_{n+1}|_w^v$. Hence, $P_IP_J$ is a monomial in $F_{n+1}|_w^v$.  

\medskip

\textbf{Case 2.} Assume $|J| < t' \le |I| < t$. Note that $v \le I, J, I \cup \{n+1\} $ and $I, J, I \cup \{n+1 \} \le w$. So, $P_I, P_J$ and $P_{I \cup \{n+1\}}$ do not vanish in $F_{n+1}|_w^v$. Since $P_{I'} P_{J'}$ vanishes in $F_n|_{\underline w}^{\underline v}$, we may proceed by taking cases on which of the following hold: $\underline v \nleq I', \underline v \nleq J', \underline w \ngeq I'$ or $\underline w \ngeq J'$. For each case we define sets $\tilde I, \tilde I', \tilde J, \tilde J'$ such that: $P_{\tilde I}P_{\tilde J} - P_{\tilde I'}P_{\tilde J'}$ is a binomial in $\init_{W_D}(F_{n+1})$, $P_{\tilde I}P_{\tilde J}$ does not vanish, and $P_{\tilde I'}P_{\tilde J'}$ vanishes in $F_{n+1}|_w^v$. It follows that $F_{n+1}|_w^v$ contains the monomial $P_{\tilde I}P_{\tilde J}$.

\smallskip

\textbf{Case 2.1.} If either $\underline v \nleq I', \underline v \nleq J'$ or $\underline w \ngeq J'$, then define $\tilde I = I, \tilde I' = I, \tilde J = J$ and $\tilde J' = J'$. So, $P_{\tilde I}P_{\tilde J}$ does not vanish in $F_{n+1}|_w^v$. However, either $v \nleq \tilde I', v \nleq \tilde J'$ or $w \ngeq \tilde J'$ holds, respectively. And so $P_{\tilde I'}P_{\tilde J'}$ vanishes in $F_{n+1}|_w^v$.

\smallskip

\textbf{Case 2.2.} If $\underline w \ngeq I'$, then define $\tilde I = I \cup \{n+1 \}, \tilde I' = I \cup \{n+1 \}, \tilde J = J$ and $\tilde J' = J'$. So, $P_{\tilde I}P_{\tilde J}$ does not vanish in $F_{n+1}|_w^v$. However, we have that $w \ngeq \tilde I'$. Hence, $P_{\tilde I'}$ vanishes in $F_{n+1}|_w^v$.

\medskip

\textbf{Case 3.} Assume that $|J| < t' \le t \le |I|$. Define $\tilde I = I \cup \{n+1\}, \tilde J = J, \tilde I' = I' \cup \{n+1\}$, and $\tilde J' = J'$. Observe that $P_{\tilde I}P_{\tilde J} - P_{\tilde I'} P_{\tilde J'}$ is a binomial in $\init_{W_D}(F_{n+1})$. Since $v \le \tilde I, \tilde J \le w$, we have that $P_{\tilde I} P_{\tilde J}$ does not vanish in $F_{n+1}|_w^v$. Since $P_{I'} P_{J'}$ vanishes in $F_{n}|_{\underline w}^{\underline v}$, we have that at least one of $\underline v \nleq I', \underline v \nleq J', \underline w \ngeq I'$ or $\underline w \ngeq J'$ holds. It follows that at least one of $\underline v \nleq I', \underline v \nleq J', \underline w \ngeq I'$ or $\underline w \ngeq J'$ holds respectively. And so $P_{\tilde I'}P_{\tilde J'}$ vanishes in $F_{n+1}|_w^v$. Therefore $P_{\tilde I}P_{\tilde J}$ is a monomial in $F_{n+1}|_w^v$.

\medskip

{\bf Case 4.} 
Assume $t'\leq |J| \le |I| < t$. Note that $v \le I, J, I \cup \{ n+1\}, J \cup \{n+1\}$ and $I, J, I \cup \{n+1 \}, J \cup \{n+1\} \le w$. So, $P_I, P_J, P_{I \cup \{n+1\}}$ and $P_{J \cup \{n+1\}}$ do not vanish in $F_{n+1}|_w^v$. Since $P_{I'} P_{J'}$ vanishes in $F_n|_{\underline w}^{\underline v}$, we may proceed similarly to Case~2 by taking cases on which of the following hold: $\underline v \nleq I', \underline v \nleq J', \underline w \ngeq I'$ or $\underline w \ngeq J'$. 

\smallskip

\textbf{Case 4.1.} If $\underline v \nleq I'$ or $\underline v \nleq J'$, then define $\tilde I = I, \tilde J = J, \tilde I' = I'$, and $\tilde J' = J'$. It follows that $v \nleq \tilde I'$ or $v \nleq \tilde J'$ holds, respectively. Hence, $P_{\tilde I} P_{\tilde J}$ is a monomial in $F_{n+1}|_w^v$.

\smallskip

\textbf{Case 4.2.} If $\underline w \ngeq I'$ or $\underline w \ngeq J'$, then define $\tilde I = I \cup \{n+1\}, \tilde J = J\cup \{n+1\}, \tilde I' = I'\cup \{n+1\}$, and $\tilde J' = J'\cup \{n+1\}$. It follows that $w \ngeq \tilde I'$ or $w \ngeq \tilde J'$ holds, respectively. Hence, $P_{\tilde I} P_{\tilde J}$ is a monomial in $F_{n+1}|_w^v$.

\medskip

\textbf{Case 5.}
Assume that $t' \leq |J| < t \le |I|$. Note that $v \le I \cup \{ n+1\}, J, J \cup \{n+1\}$ and $I \cup \{n+1 \}, J, J \cup \{n+1\} \le w$. So, $P_{I \cup \{n+1\}}, P_J$ and $P_{J \cup \{n+1\}}$ do not vanish in $F_{n+1}|_w^v$. Since $P_{I'} P_{J'}$ vanishes in $F_n|_{\underline w}^{\underline v}$, we may proceed similarly to Case~2 by taking cases on which of the following hold: $\underline v \nleq I', \underline v \nleq J', \underline w \ngeq I'$ or $\underline w \ngeq J'$. 

\smallskip

\textbf{Case 5.1.} If either $\underline v \nleq I', \underline v \nleq J'$ or $\underline w \ngeq I'$, then define $\tilde I = I \cup \{n+1 \}, \tilde J = J, \tilde I' = I' \cup \{n+1 \}$, and $\tilde J' = J'$. It follows that either $v \nleq \tilde I', v \nleq \tilde J'$ or $w \ngeq \tilde I'$ holds, respectively. Hence, $P_{\tilde I} P_{\tilde J}$ is a monomial in $F_{n+1}|_w^v$.

\smallskip

\textbf{Case 5.2.} If $\underline w \ngeq J'$, then define $\tilde I = I \cup \{n+1 \}, \tilde J = J \cup \{n+1\}, \tilde I' = I' \cup \{n+1 \}$, and $\tilde J' = J'\cup \{n+1\}$. It follows that $w \ngeq \tilde J'$, hence, $P_{\tilde I} P_{\tilde J}$ is a monomial in $F_{n+1}|_w^v$.

\medskip

\textbf{Case 6.} Assume that $t \le |J|$. Define
\[
\tilde I=I\cup \{n+1\}, \ 
\tilde J=J\cup \{n+1\}, \ 
\tilde I'=I'\cup \{n+1\} \text{ and }
\tilde J'=J'\cup \{n+1\}.
\]
Then $P_{\tilde I}P_{\tilde J}-P_{\tilde I'}P_{\tilde J'}$ is a binomial in  $\init_{W_D}(F_{n+1}).$ Since $\underline v \leq I, J\leq \underline w$, then $v\leq \tilde I, \tilde J\leq w$ so $P_{\tilde I}P_{\tilde J}$ does not vanish in $F_{n+1}|_w^v$. Since $P_{I'}P_{J'}$ vanishes in $F_{n}|_{\underline w}^{\underline v}$, we have that at least one of $\underline v \nleq I', \underline v \nleq J', \underline w \ngeq I'$ or $\underline w \ngeq J'$ holds. It follows that at least one of $\underline v \nleq I', \underline v \nleq J', \underline w \ngeq I'$ or $\underline w \ngeq J'$ holds respectively. And so $P_{\tilde I'}P_{\tilde J'}$ vanishes in $F_{n+1}|_w^v$. Therefore $P_{\tilde I}P_{\tilde J}$ is a monomial in $F_{n+1}|_w^v$. \qed \color{black}

\end{proof}

For the following lemma, recall that for any pair of permutations $(v,w) \in S_{n+1} \times S_{n+1}$, we denote $v_t = w_{t'} = n+1$ and $v_s = w_{s'} = n$. In addition we write $(w_1, \dots, w_k)\!\!\uparrow$ for the ordered list whose elements are $\{w_1, \dots, w_k \}$ taken in increasing order.

\begin{lemma}\label{comp}
Let $v,w \in S_{n+1}$ with $\underline v \le \underline w$. If $\init(I_{n+1})|_{T_w^v}$ is monomial-free then $(v,w)$ is a compatible pair.
\end{lemma}

\begin{proof}
Note that if $t' = t$ then $(v,w)$ is compatible. Since $\underline v \le \underline w$, we may assume that $t'<t$. We have that $\init(I_{n+1})|_{T_w^v}$ is monomial-free, 
so by Lemma \ref{non toric}, we have that $\init(I_n)|_{T_{\underline w}^{\underline v}}$ is monomial-free. We prove that if $(v, w) \in S_{n+1} \times S_{n+1}$ is not compatible and $v < w$ then $\init(I_{n+1})|_{T_w^v}$ contains a monomial.
Let $v_\ell=n-1=w_{\ell'}$.

\medskip

\textbf{Case 1.} Assume that $t<s'$. By compatibility we have $t'\leq t<s'\leq s$. Suppose that $s' = s$ and $\ell' > s$. Let
\[
I = \{w_1, \dots, w_{s'} \} = \{i_1 < i_2 < \dots < i_{s'-2} < n < n+1\}, \quad
J = \{i_1 < \dots < i_{s'-3} < n-1 \}.
\]
Note that $\ul v \le \ul w$ and $s = s' < \ell'$ therefore $i_{s'-2} < n-1$. And so we have $v \le I, J \le w$. Let
\[
I' = \{i_1 < \dots <i_{s'-3} < n-1 < n < n+1 \}, \quad
J' = \{i_1 < \dots < i_{s'-2}\}.
\]
By construction it is clear that $P_IP_J - P_{I'}P_{J'}$ is a binomial in $\init(I_{n+1})$. Since $n-1 \notin \{w_1, \dots, w_{s'} \}$, it follows that $I' \not\le w$. And so $P_IP_J$ is a monomial in $\init(I_{n+1})|_{T_w^v}$. So we may now assume that either $\ell'\leq s$ or $s' < s$. However, if $s' < s$ then, by induction, we have $\ell'\leq s$. So we assume $\ell' \leq s$.

\medskip

{\bf Case 1.1.} Let $\ell'>s'$. Since $t'\leq t<s'$ and $t'\neq t$, we have $t'<s'-1$. Take 
\vspace{-2mm}
\[
I=(w_1, \ldots, w_{s'})\!\!\uparrow, \quad 
J=\begin{cases} (w_1, \ldots, w_{t'-1}, w_{\ell'})\!\!\uparrow & \text {if}~ t'>1, \\ w_{\ell'} & \text {if}~ t'=1, \end{cases}
\]
\vspace{-3mm}
\[
I'=\begin{cases}(w_1, \ldots, w_{t'-1}, w_{\ell'}, w_{t'}, w_{t'+2}, \ldots, w_{s'})\!\!\uparrow & \text {if}~ t'>1, \\  (w_{\ell'}, w_{t'}, w_{t'+2}, \ldots, w_{s'})\!\!\uparrow & \text {if}~ t'=1 \end{cases} \text{ and }
\]
\vspace{-3mm}
\[
J'=\begin{cases}(w_1, \ldots, w_{t'-1}, w_{t'+1})\!\!\uparrow & \text {if}~ t'>1, \\ w_{t'+1} & \text {if}~ t'=1.  \end{cases}
\]
It is clear from the construction that $P_IP_J-P_{I'}P_{J'}$ is a binomial in $\init(I_{n+1})$ and $v\leq I, J\leq w$. So $P_IP_J$ does not vanish in $\init(I_{n+1})|_{T_w^v}$. Since $I'\nleq w$, then $P_{I'}$ vanishes in $\init(I_{n+1})|_{T_w^v}$. Therefore, $P_IP_J$ is a monomial in $\init(I_{n+1})|_{T_w^v}$.
 
\medskip

\textbf{Case 1.2.}
Let $\ell'<s'$. Let $k= \max\{\ell', t\}$ and $v_{r}=max\{v_i: 1\leq i\leq k, i\neq t\}$. Now define 
\vspace{-2mm}
\[
I=(v_1, \ldots, v_{k})\!\!\uparrow, \quad 
J=(v_1, \dots, v_{r-1}, v_{r+1}, \dots, v_{k-1}, n-1)\!\!\uparrow,
\]
\vspace{-3mm}
\[
I'=(v_1, \ldots, v_{r-1}, v_{r+1}, \ldots, v_{k},  n-1,)\!\!\uparrow
\text{ and } 
J'= (v_1, \dots, v_{k-1})\!\!\uparrow
\]
Consider the tableaux for $P_IP_J$ and $P_{I'}P_{J'}$. Note that all rows are the same except for the $(k-1)^{th}$ row, and in this row we interchange $n-1$ and $v_r$. Since $k<s'\leq s$, $n-1\notin I$ and $P_IP_J-P_{I'}P_{J'}$ is a binomial in $\init(I_{n+1})$, it follows that $v\leq I, J\leq w$ and so $P_IP_J$ does not vanish in $\init(I_{n+1})|_{T_w^v}$.
 Since $k<s'$ and $n-1\in I'$, we see that $I'\nleq w$. So $P_{I'}$ vanishes in $\init(I_{n+1})|_{T_w^v}$.  

\medskip
{\bf Case 2.} Assume that $t'>s$. Then we have $t\geq t'>s\geq s'$. Let $v_r=max\{v_i: 1\leq i\leq t', i\neq s\}$. We define 
\vspace{-2mm}
\[
I=\begin{cases}(v_1, \ldots, v_{s-1}, n+1, v_{s+1}, \ldots, v_{t'})\!\!\uparrow & \text {if}~ s>1, \\    
(n+1, v_{s+1}, \ldots, v_{t'})\!\!\uparrow & \text {if}~ s=1,  \end{cases} \quad J=(v_1, \ldots, v_{t'-1})\!\!\uparrow,
\]
\vspace{-3mm}
\[
I'=(v_1, \ldots,v_{r-1}, v_{r+1}, \ldots, v_{t'}, n+1)\!\!\uparrow \text{ and }
J'=\begin{cases}(v_1, \ldots, v_{s-1}, v_{s+1}, \ldots, v_{t'-1}, v_r) \!\!\uparrow  & \text {if}~ s>1, \\ (v_{s+1}, \ldots, v_{t'-1}, v_r) \!\!\uparrow & \text {if}~ s=1. \end{cases}
\]
By construction, $P_IP_J-P_{I'}P_{J'}$ is a binomial in $\init(I_{n+1})$. Since $n-1, n \notin J'$ and $t'>s$, we have $v\nleq J'$ and $P_IP_J$ is a monomial in $\init(I_{n+1})|_{T_w^v}$.
\color{black}

\medskip

{\bf Case 3.} Assume that there exists $t'\leq k < t$ such that $v_k>v_{k+1}$.
We take
\vspace{-2mm}
\[
I=\begin{cases}(v_1, \ldots, v_{k-1}, v_{k+1}, n)\!\!\uparrow & \text {if}~ k>1, \\   (v_{k+1}, n)\!\!\uparrow & \text {if}~ k=1,  \end{cases} \quad J=(v_1, \ldots, v_{k})\!\!\uparrow,
\]
\vspace{-2mm}
\[
I'=\begin{cases}(v_1, \ldots,v_{k-1}, v_{k}, n)\!\!\uparrow & \text {if}~ k>1, \\ (v_{k}, n)\!\!\uparrow & \text {if}~ k=1, \end{cases} \text{ and } J'=(v_1, \ldots, v_{k-1}, v_{k+1}) \!\!\uparrow.
\]
 
Then $P_IP_J-P_{I'}P_{J'}$ is a binomial in $\init(I_{n+1})$. Since $v_{k+1}>v_k$, we have $v\leq I, J$ and $v\nleq J'$. Now we show that $I\leq w$. Since $\underline v \leq \underline w$ and $v_{k+1}>v_k$, we have 
\vspace{-2mm}
$$(v_1, v_2, \ldots, v_{k-1}, v_{k+1})\!\!\uparrow \leq (w_1, \ldots, w_{t'-1}, w_{t'+1}, \ldots, w_k, w_{k+1})\!\!\uparrow .$$
Then $(v_1, v_2, \ldots, v_{k-1}, v_{k+1}, n)\!\!\uparrow \leq (w_1, \ldots, w_k, w_{k+1})\!\!\uparrow$.   Then we see that $P_IP_J$ is non-zero and $P_{J'}$ vanishes in $\init(I_n)|_{T_w^v}$. Hence, $P_IP_J$ is a monomial in $\init(I_{n+1})|_{T_w^v}$. 
 
\medskip

{\bf Case 4.} Assume that there exists $t'< k\leq t$ with $w_k<w_{k+1}$.  In this case, we choose:
\vspace{-2mm}
\[
I=(w_1, \ldots, w_{k})\!\!\uparrow, \quad
J=\begin{cases}(w_1, \ldots, w_{t'-1}, w_{t'+1}, \ldots, w_{k-1}, w_{k+1})\!\!\uparrow & \text{if}~ t'>1, \\
(w_2, \ldots, w_{k-1}, w_{k+1})\!\!\uparrow & \text{if}~ t'=1,
\end{cases} 
\]
\vspace{-2mm}
\[
I'=(w_1, \ldots,w_{k-1}, w_{k+1})\!\!\uparrow \text{ and } J'=\begin{cases}
(w_1, \!\!\ldots, w_{t'-1}, w_{t'+1}, \ldots, w_{k-1}, w_{k})\!\!\uparrow & \text{if}~ t'> 1, \\
(w_2, \ldots, w_{k-1}, w_{k}) & \text{if}~ t'=1. \end{cases}
\]
 
It is easy to see that $P_IP_J-P_{I'}P_{J'}$ is a binomial in $\init(I_{n+1})$.
Since $w_{k}<w_{k+1}$, we have $v\leq I\leq w$ and $J\leq w$. Now we show that $v\leq J$. Since $\underline v \leq \underline w$ and $w_{k+1}>w_k$, we have $(v_1, v_2, \ldots, v_{k-1})\!\!\uparrow \leq (w_1, \ldots, w_{t'-1}, w_{t'+1}, \ldots, w_{k-1}, w_{k+1})\!\!\uparrow$. Note that $P_IP_J$ is non-zero and $P_{I'}$ vanishes in $\init(I_{n+1})|_{T_w^v}$. Then $P_IP_J$ is a monomial in $\init(I_{n+1})|_{T_w^v}$. 
 
Hence, we conclude that if $(v, w)$ is not compatible, then $(v, w) \not\in \MT_{n+1}$, as desired. 
\end{proof}

\begin{lemma}\label{lem:compatible_is_monomial_free}
Let $v,w \in S_{n+1}$. If $\init(I_n)|_{T_{\underline w}^{\underline v}}$ is monomial-free and $(v,w)$ is a compatible pair then $\init(I_{n+1})|_{T_w^v}$ is monomial-free.
\end{lemma}

\begin{proof}
We proceed by double induction, first on $n$ and secondly on the dimension of the Richardson variety. By Lemma~\ref{comp}, we may assume that for all $v,w \in S_k$ where $k \le n$ we have $\init(I_k)|_{T_w^v}$ is monomial-free if and only if $(v,w)$ is a compatible pair and $\init(I_{k-1})|_{T^{\underline v}_{\underline w}}$ is monomial-free.
Fix $v,w \in S_{n+1}$ and let $P_IP_J - P_{I'} P_{J'}$ be a binomial in the ideal $\init(I_{n+1})$ and assume that $P_IP_J$ does not vanish in $\init(I_{n+1})|_{T_w^v}$. Without loss of generality we assume that $|I| = |I'| \ge |J| = |J'|$. We will show that $P_{I'}P_{J'}$ does not vanish in $\init(I_{n+1})|_{T_w^v}$ by taking cases on $t, t'$.

\textbf{Case 1.} Assume that $t,t'$ both lie in one of the following sets: $\{1, \dots, |J|\},  \{|J|+1, \dots, |I| \}$ or $\{|I|+1, \dots, n+1 \}$. For each of these cases we deduce immediately which sets $I, I', J, J'$ contain $n+1$. By removing $n+1$ from these sets we obtain $\underline I, \underline I', \underline J, \underline J'$. We have that $P_{\underline I} P_{\underline J} - P_{\underline I'}P_{\underline J'}$ is a binomial in $\init(I_n)$. Since $P_I P_J$ does not vanish in $\init(I_{n+1})|_{T_w^v}$, by construction we have that $P_{\underline I}P_{\underline J}$ does not vanish in $\init(I_n)|_{T_{\underline w}^{\underline v}}$. Since $(\underline v, \underline w) \in \MT_n$ we have that $P_{\underline I'}P_{\underline J'}$ does not vanish in $\init(I_n)|_{T_{\underline w}^{\underline v}}$. By construction it follows that $P_{I'} P_{J'}$ does not vanish in $\init(I_{n+1})|_{T_w^v}$.

\textbf{Case 2.} Assume that $t' \in \{|J|+1, \dots, |I| \}$ and $t \in \{|I|+1, \dots, n+1\}$.

\textbf{Case 2.1} Assume that $n+1 \in I$.
By Case i. of the claim on page~\pageref{claim:compatible_v_w}, 
 we have that $v' = (v_1, \dots, v_{t-2}, v_t, v_{t-1}, v_{t+1}, \dots, v_{n+1}) \le I$. Since $|J| < |I|$ we have $v' \le J$. By Lemma~\ref{lem:inductive_vw} we have that $(v',w) \in \MT_{n+1}$ and so by induction on the dimension $P_{I'}P_{J'}$ does not vanish in $\init(I_{n+1})|_{T_w^{v'}}$. Therefore, $v < v' \le I', J' \le w$ and so $P_{I'}P_{J'}$ does not vanish in $\init(I_{n+1})|_{T_w^v}$.

\textbf{Case 2.2} Assume that $n+1 \notin I$. By Case ii. of the claim on page~\pageref{claim:compatible_v_w}, we have that $I \le w' = (w_1, \dots, w_{t'-1}, w_{t'+1}, w_{t'}, v_{t+1}, \dots, w_{n+1})$. Since $|J| < |I|$ we have $J \le w'$. By Lemma~\ref{lem:inductive_vw} we have that $(v,w') \in \MT_{n+1}$ and so by induction on the dimension $P_{I'}P_{J'}$ does not vanish in $\init(I_{n+1})|_{T_{w'}^{v}}$. Therefore, $v \le I', J' \le w' < w$ and so $P_{I'}P_{J'}$ does not vanish in $\init(I_{n+1})|_{T_w^v}$.

\textbf{Case 3.} Assume that $t' \in \{1, \dots, |J| \}$ and $t \in \{|J|+1, \dots, |I| \}$. This case is identical to Case 2, where we can use a similar argument to show that either $v' \le J$ if $n+1 \in J$ or $J \le w'$ if $n+1 \notin J$.

\textbf{Case 4.} Assume that $t' \in \{1, \dots, |J| \}$ and $t \in \{|I|+1, \dots, n+1 \}$. If $t' < |J|$ or $t > |I|+1$ then we can use the construction of $v', w'$ as above and conclude the result by induction on the dimension of $X_w^v$. So we may assume that $t' = |J|$ and $t = |I|+1$. In this case we write $I = \{i_1 < \dots < i_{|I|} \}$ and $J = \{j_1 < \dots < j_{|J|} \}$. 

Let $(v,w)_i^j$ be the maximum block of $(v,w)$. Then we have by Proposition~\ref{prop:block_max_order} that $v_i < \dots < v_t = n+1$ and by Proposition~\ref{prop:unified} we have $v_k < v_i$ for all $k < i$. Since $t' = |J|$ it follows that $i \le |J|$. And so ordering the first $t$ elements of $v$ we get
\[
\{v_1, \dots, v_t \} = \{\widehat v_1 < \dots < \widehat v_t \} =  \{\widehat v_1 < \dots < \widehat v_{i-1} < v_i < \dots < v_t \}.
\]
And so for all $k$ we have $\widehat v_k \le i_k$ and $\widehat v_k \le j_k$, i.e.~the ordering of the first $|I|$ elements of $v$ coincides with ordering of the first $|J|$ elements of $v$. Since the tableaux representing $P_IP_J$ and $P_{I'} P_{J'}$ are row-wise equal, it follows that $v \le I'$ and $v \le J'$.

Next we show that $I' \le w$ and $J' \le w$. By Proposition~\ref{prop:block_max_order} we have $w_{t'} > \dots > w_j$ and by Proposition~\ref{prop:unified} we have that $w_k < w_j$ for all $k > j$. 

\smallskip

\textbf{Claim.} For any $K \subseteq [n+1]$: $K \le w$ if and only if $K^c := [n+1] \backslash K \ge w w_0 = (w_{n+1}, w_n, \dots, w_1)$.

To prove the claim, it is an easy observation that $K \le w$ if and only if $(K^\uparrow, (K^c)^\uparrow) \le w$. Then $(K^\uparrow,  (K^c)^\uparrow) \le w$ if and only if $(K^\uparrow,  (K^c)^\uparrow) w_0 \ge w w_0$. Explicitly we have $(K^\uparrow,  (K^c)^\uparrow) w_0 = ((K^c)^\downarrow,  K^\downarrow)$. And so $(K^\uparrow,  (K^c)^\uparrow) w_0 \ge w w_0$ if and only if $K^c \ge w w_0$. This completes the proof of the claim.

Since the tableaux representing $P_I P_J$ and $P_{I'} P_{J'}$ are row-wise equal, it follows that the tableaux representing $P_{I^c} P_{J^c}$ and $P_{I'^c} P_{J'^c}$ are also row-wise equal. By the claim above we have that $ w w_0 \le I'^c$ and $ w w_0 \le J'^c$. And so we have $I' \le w$ and $J' \le w$.
\end{proof}

\section{Toric degenerations 
}\label{sec:toric}
Recall the ideals $\init(I_n)|_{T_w^v}$, $\init(I(X_w^v))$ and $\ker{(\phi_n|_w^v)}$ 
from Sections~\ref{subsec:flag} and \ref{subsec:X}.
In this section, we will study  
the relationships between these ideals with the goal of understanding the initial ideal $\init(I(X_w^v))$. When this ideal is toric, we obtain a toric degeneration of the Richardson variety $X_w^v$ inside the flag variety.
We recall, by Theorem~\ref{thm:Pure}, that the initial ideal $\init(I_n)$ is quadratically generated and is the kernel of the monomial map $\phi_n$ in \eqref{eqn:monomialmapflags}.
We will see that if the ideal $\init(I_n)|_{T_w^v}$ is monomial-free then $\init(I(X_w^v))$ is quadratically generated.
Furthermore, we prove that if $\init(I_n)|_{T_w^v}$ is monomial-free then 
$\init(I_n)|_{T_w^v} = \init(I(X_w^v))$ and
$\init(I_n)|_{T_w^v}$ is a toric ideal.
We prove this by showing that if $\init(I_n)|_{T_w^v}$ is monomial-free then it is equal to the ideal $\ker(\phi_n|_w^v)$.

\begin{lemma}\label{lem:J_1=J_3_binomial_flag}\label{lem:J_1_subset_J_2_flag} 
We have the following:

\begin{itemize}
    \item[{\rm (i)}] The ideal $\init(I_n)|_{T_w^v}$ is monomial-free if and only if it coincide with $\ker{(\phi_n|_w^v)}$. 
    \item[{\rm (ii)}] $\init(I_n)|_{T_w^v}\subseteq\init(I(X_w^v))$.
\end{itemize}
\end{lemma}

\begin{proof}
By Theorem~\ref{thm:Pure}, there exists a set $G$ of quadratic binomials which generate the initial ideal $\init(I_n)$ and by Lemma~\ref{lem:elim_ideal_gen_set}(ii), the ideal $\init(I_n)|_{T_w^v}$ is generated by $G|_{T_w^v}$. See \eqref{eq:G|T}.

(i) First note that $\phi_n|_w^v$ is a monomial map, hence its kernel does not contain any monomials. So, if the ideal $\init(I_n)|_{T_w^v}$ contains a monomial then it is not equal to $\ker{(\phi_n|_w^v)}$.
Now assume that the ideal $\init(I_n)|_{T_w^v}$ does not contain any monomials, therefore the set $G|_{T_w^v}$ does not contain any monomials.
Since all binomials $m_1 - m_2 \in G|_{T_w^v}$ lie in $\init(I_n)$ and contain only the non-vanishing Pl\"ucker variables $P_J$ for $J \in T_w^v$, therefore $m_1 - m_2 \in \ker{(\phi_n|_w^v)}$. And so we have $\init(I_n)|_{T_w^v} \subseteq \ker{(\phi_n|_w^v)}$. Thus the proof of (i) follows.

(ii) Since $\init(I_n)|_{T_w^v} = \langle G|_{T_w^v}\rangle$, we take $\hat g \in G|_{T_w^v}$. So we have that $g \in G$ and 
there exists $f \in I(X_w^v)$ such that $\init(f) = g$. 
The terms of $\hat g$ are precisely the non-vanishing terms of the initial terms of $f$. Therefore, $\hat g = \init(\hat f) \in \init(I(X_w^v))$.
This completes the proof of lemma.
\end{proof}

\begin{theorem}\label{thm:toric_degen}
If the ideal $\init(I_n)|_{T_w^v}$ is monomial-free, then 
$\init(I(X_w^v))$ is a toric ideal and it provides a toric degeneration of the Richardson variety $X_w^v$. 
\end{theorem}

\begin{proof}
Let us consider $\MCM \subseteq R:=\KK[P_J: J\in T_w^v]$ be a collection of monomials which are linearly independent in $R/ \init(I(X_w^v))$. If the image of $\MCM$ in $R/ \init(I_n)|_{T_w^v}$ is a linearly dependent subset, then we have $\sum_{m \in M} c_m m \in \init(I_n)|_{T_w^v}$ for some $c_m \in \mathbb K$. So by Lemma~\ref{lem:J_1_subset_J_2_flag} we have  $\sum_{m \in M} c_m m \in \init(I(X_w^v))$,
and so the image of $\MCM$ in $R/\init(I(X_w^v))$ is linearly dependent, a contradiction.
Moreover, since  $\init(I_n)|_{T_w^v}$ is monomial-free we have that $\init(I_n)|_{T_w^v} = \ker{(\phi_n|_w^v)}$. 
Hence, for all $d \ge 1$, any standard monomial basis for $R/\init(I(X_w^v))$ of degree $d$ is linearly independent in $R/ \init(I_n)|_{T_w^v} = R/ \ker{(\phi_n|_w^v)}$. 
Note that $\ker(\phi_n|_w^v)$ is generated by binomials which correspond to pairs of row-wise equal tableaux, whose columns $I$ satisfy $v \le I \le w$. It is easy to see that every tableau is row-wise equal to a unique semi-standard Young tableau. Therefore, the semi-standard Young tableaux $SSYT_d(v,w)$ form a monomial basis for the degree $d$ part of $R / \ker(\phi_n|_w^v)$.

Note that any Gr\"obner degeneration gives rise to a flat family, so the Hilbert polynomials of all fibers are identical. By Theorem~\ref{thm:ssyt_fl_diag}, the semi-standard Young tableaux $SSYT_d(v,w)$ form a standard monomial basis for $R / I(X_w^v)$.
Thus, the dimension of the degree $d$ part of $R / \init(I(X_w^v)$ is equal to $|SSYT_d(v,w)|$. And so $\init(I_n)|_{T_w^v} = \ker(\phi_n|_w^v) = \init(I(X_w^v))$. 
\end{proof}

\begin{remark}
\label{sec:compareKim's}
Our methods can be used to produce other toric degenerations of Richardson varieties with respect to different weight vectors. For example, in \cite{kim2015richardson}, Kim considers a weight on the polynomial ring $\KK[x_{i,j}]$ such that the leading term of any minor $\varphi_n(P_J)$, see \eqref{eq:map}, is the antidiagonal term. 
Explicitly, if $M'$ is the weight on $\KK[x_{i,j}]$ then \[
\init_{M'}(\varphi_n(P_J)) = x_{1, j_t} x_{2, j_{t-1}} \dots x_{t, j_1}\text{ for every } J = \{j_1 < \dots < j_t\}.
\]
We write ${\bf w}_{M'}$ for the weight on the ring $\KK[P_J : J \subseteq [n]]$ induced by $M'$.
In \cite{kim2015richardson}, the pairs of permutations $(v,w)$ are labelled by collections of so-called \emph{pipe dreams}. 
Each pair of reduced pipe dreams associated to $(v,w)$ gives rise to a face of the Gelfand-Tsetlin polytope. 
If $(v,w)$ is labelled by a unique pair of reduced pipe dreams, then the initial ideal $\init_{{\bf w}_{M'}}(I(X_w^v))$ is toric.
However, if there are multiple pairs of pipe dreams associated to $(v,w)$, then this method cannot differentiate between toric and non-toric initial ideals. For each $(v,w) \in S_4 \times S_4$ we have calculated the ideals $\init_{{\bf w}_{M'}}(I_4)|_{T_w^v}$. 
We have confirmed that Theorem~\ref{thm:toric_degen} holds in all these cases, i.e.~if $\init_{{\bf w}_{M'}}(I_4)|_{T_w^v}$ is monomial-free then the initial ideal $\init_{{\bf w}_{M'}}(I(X_w^v))$ is toric. Our calculations are displayed in Table~\ref{table:flag_4}. The symbol $*$ appears in the table beside pairs of permutations for which the description by pipe dreams does not determine whether the corresponding ideal is toric or non-toric.

In many cases, it is possible to give an explicit description of the polytopes associated to toric degenerations. Any toric variety whose ideal is of the form $\init_{{\bf w}_{M'}}(I(X_w^v))$, for some $v$ and $w$, is a toric subvariety of the toric variety associated to the Gelfand-Tsetlin polytope. Therefore, the toric polytope associated to $\init_{{\bf w}_{M'}}(I(X_w^v))$ is a face of the Gelfand-Tsetlin polytope. Suppose that $\init_{{\bf w}_{M'}}(I(X_w^v))$ is toric. On the one hand, if there is a unique pair of reduced pipe dreams associated to $(v,w)$, then this face of the Gelfand-Tsetlin polytope is determined uniquely and is described in terms of Gelfand-Tsetlin patterns in \cite{kim2015richardson}. On the other hand, if there does not exist a unique pair of pipe dreams associated to $(v,w)$, then determining the face of the Gelfand-Tsetlin polytope is a difficult computational task. 

Computing polytopes of toric degenerations of the flag variety  $\Flag_n$ for large $n$ is already very difficult. For example, all toric degenerations via Gr\"obner degenerations have been calculated up to $\Flag_5$, see \cite{bossinger2017computing}. One approach to understand these toric polytopes, and by extension their faces, is to first consider the Grassmannian and its toric degenerations studied in \cite{OllieFatemeh, bossinger2021families}. The vertices of the corresponding toric polytopes can be read directly from the monomial map, analogous to (\ref{eqn:monomialmapflags}). In \cite{clarke2020combinatorial}, the authors study these polytopes using \textit{combinatorial mutations} which preserve many important properties of the polytope, such as its \textit{Ehrhart function}. However, in the forthcoming work \cite{clarke2021combinatorial, clarke2022}, the authors note that the monomial map (\ref{eqn:monomialmapflags}) does not immediately give rise to the toric polytope. Instead, they give a combinatorial analogue to: embedding products of projective varieties into higher dimensional projective spaces, for polytopes. We give an example of this procedure below.
\end{remark}

\begin{table}
    \centering  
    \resizebox{0.8\textwidth}{!}{
    \begin{tabular}{|c|c|c|}
        \hline
        ((1, 2, 3, 4), (1, 4, 2, 3)) \textcolor{white}{$*$}   &((2, 3, 1, 4), (4, 3, 1, 2)) \textcolor{white}{$*$}   &((1, 2, 3, 4), (1, 2, 4, 3)) \textcolor{white}{$*$} \\
        ((1, 2, 3, 4), (1, 4, 3, 2)) \textcolor{white}{$*$}   &((2, 3, 1, 4), (4, 3, 2, 1)) \textcolor{white}{$*$}   &((1, 2, 3, 4), (1, 3, 2, 4)) $*$ \\
        ((1, 2, 3, 4), (3, 1, 2, 4)) \textcolor{white}{$*$}   &((2, 3, 4, 1), (4, 2, 3, 1)) $*$                      &((1, 2, 3, 4), (2, 1, 3, 4)) \textcolor{white}{$*$} \\
        ((1, 2, 3, 4), (3, 2, 1, 4)) \textcolor{white}{$*$}   &((2, 3, 4, 1), (4, 3, 2, 1)) \textcolor{white}{$*$}   &((1, 2, 3, 4), (2, 1, 4, 3)) \textcolor{white}{$*$} \\
        ((1, 2, 3, 4), (4, 1, 2, 3)) \textcolor{white}{$*$}   &((3, 1, 2, 4), (4, 1, 2, 3)) \textcolor{white}{$*$}   &((1, 2, 4, 3), (2, 1, 4, 3)) $*$ \\
        ((1, 2, 3, 4), (4, 1, 3, 2)) \textcolor{white}{$*$}   &((3, 1, 2, 4), (4, 1, 3, 2)) \textcolor{white}{$*$}   &((1, 3, 4, 2), (1, 4, 3, 2)) $*$ \\
        ((1, 2, 3, 4), (4, 2, 1, 3)) \textcolor{white}{$*$}   &((3, 1, 2, 4), (4, 2, 1, 3)) \textcolor{white}{$*$}   &((1, 4, 2, 3), (1, 4, 3, 2)) $*$ \\
        ((1, 2, 3, 4), (4, 3, 1, 2)) \textcolor{white}{$*$}   &((3, 1, 2, 4), (4, 3, 1, 2)) \textcolor{white}{$*$}   &((2, 1, 3, 4), (2, 1, 4, 3)) \textcolor{white}{$*$} \\
        ((1, 2, 3, 4), (4, 3, 2, 1)) \textcolor{white}{$*$}   &((3, 1, 2, 4), (4, 3, 2, 1)) \textcolor{white}{$*$}   &((2, 3, 1, 4), (3, 2, 1, 4)) \textcolor{white}{$*$} \\
        ((1, 3, 2, 4), (1, 4, 2, 3)) $*$                      &((3, 1, 4, 2), (4, 1, 3, 2)) \textcolor{white}{$*$}   &((2, 3, 4, 1), (2, 4, 3, 1)) $*$\\
        ((1, 3, 2, 4), (1, 4, 3, 2)) $*$                      &((3, 2, 1, 4), (4, 2, 1, 3)) \textcolor{white}{$*$}   &((2, 3, 4, 1), (3, 2, 4, 1)) $*$\\
        ((2, 1, 3, 4), (3, 1, 2, 4)) \textcolor{white}{$*$}   &((3, 2, 1, 4), (4, 3, 1, 2)) \textcolor{white}{$*$}   &((3, 1, 2, 4), (3, 2, 1, 4)) \textcolor{white}{$*$} \\
        ((2, 1, 3, 4), (3, 2, 1, 4)) \textcolor{white}{$*$}   &((3, 2, 1, 4), (4, 3, 2, 1)) \textcolor{white}{$*$}   &((3, 4, 1, 2), (3, 4, 2, 1)) $*$ \\
        ((2, 1, 3, 4), (4, 1, 2, 3)) \textcolor{white}{$*$}   &((3, 2, 4, 1), (4, 2, 3, 1)) $*$                      &((3, 4, 1, 2), (4, 3, 1, 2)) \textcolor{white}{$*$} \\
        ((2, 1, 3, 4), (4, 1, 3, 2)) \textcolor{white}{$*$}   &((3, 2, 4, 1), (4, 3, 2, 1)) \textcolor{white}{$*$}   &((3, 4, 2, 1), (4, 3, 2, 1)) \textcolor{white}{$*$} \\
        ((2, 1, 3, 4), (4, 2, 1, 3)) \textcolor{white}{$*$}   &((4, 1, 2, 3), (4, 3, 1, 2)) \textcolor{white}{$*$}   &((4, 1, 2, 3), (4, 1, 3, 2)) \textcolor{white}{$*$} \\
        ((2, 1, 3, 4), (4, 3, 1, 2)) \textcolor{white}{$*$}   &((4, 1, 2, 3), (4, 3, 2, 1)) \textcolor{white}{$*$}   &((4, 1, 2, 3), (4, 2, 1, 3)) \textcolor{white}{$*$} \\
        ((2, 1, 3, 4), (4, 3, 2, 1)) \textcolor{white}{$*$}   &((4, 2, 1, 3), (4, 3, 1, 2)) \textcolor{white}{$*$}   &((4, 2, 3, 1), (4, 3, 2, 1)) $*$ \\
        ((2, 3, 1, 4), (2, 4, 1, 3)) \textcolor{white}{$*$}   &((4, 2, 1, 3), (4, 3, 2, 1)) \textcolor{white}{$*$}   &((4, 3, 1, 2), (4, 3, 2, 1)) \textcolor{white}{$*$} \\
        ((2, 3, 1, 4), (4, 2, 1, 3)) \textcolor{white}{$*$}   & &\\
        \hline
    \end{tabular}
    } \\ \caption{The list of pairs of permutations $(v,w)$ leading to toric degenerations of $X_w^v$ with respect to the antidiagonal term order in $\Flag_4$.}
    \label{table:flag_4}
\end{table}

\begin{example}\label{example:richardson_polytope}
Let $v = (2,3,4,1)$ and $w = (4,2,3,1)$. Let us consider the antidiagonal term order from Remark~\ref{sec:compareKim's}.
We have that $((2,3,4,1),(4,2,3,1))$ is contained in Table~\ref{table:flag_4}, and so the Richardson variety $X_w^v$ degenerates to a toric variety, which we will call $T$.
However, the entry has an asterisk beside it in Table~\ref{table:flag_4}, 
so it is not possible to immediately determine the polytope associated to this degeneration of $X_w^v$ using pipe dreams from \cite{kim2015richardson}. 
Let us calculate the polytope associated to this degeneration following \cite{clarke2021combinatorial, clarke2022}.

Recall that, via the Pl\"ucker embedding, the Flag variety $\Flag_4$ is a subvariety of the product of projective spaces $\PP^3 \times \PP^5 \times \PP^3$. The coordinates for each projective space are $P_I$ where $I \subseteq [n]$ is a fixed size. For instance $[P_1, P_2, P_3, P_4]$ are the homogeneous coordinates for the first copy of $\PP^3$. The Richardson variety $X_w^v$ and its toric degeneration $T$ naturally live in the product of projective spaces $\PP^{2} \times \PP^{1} \times \PP^{0}$, with coordinates given by the non-vanishing variables which are $
[P_2, P_3, P_4], [P_{23}, P_{24}]$ and $[P_{234}]$, respectively. The ideal of $T \subseteq \PP^2 \times \PP^1 \times \PP^0$ is the kernel of the monomial map $\varphi_4$, see Remark~\ref{sec:compareKim's}, restricted to the non-vanishing variables. We can write $\varphi_4$ as the following integer matrix, note that only the non-zero rows are included:
\vspace{-1mm}\[
A = \ \
\bordermatrix{
& P_2 & P_3 & P_4 & P_{23} & P_{24} & P_{234} \cr
x_2 &1& & & & & \cr
x_3 & &1& &1& & \cr
x_4 & & &1& &1&1\cr
y_2 & & & &1&1& \cr
y_3 & & & & & &1\cr
z_2 & & & & & &1\cr
}
.
\]
The image of $\PP^2 \times \PP^1 \times \PP^0$ under the Segre embedding is a toric subvariety $Y \subseteq \PP^5$. We label the coordinates for $\PP^5$ by the corresponding products of variables: $P_2P_{23}P_{234}, P_{2}P_{24}P_{234}$ and so on.
The monomial map, with matrix $S$, associated to $Y$ sends each coordinate of $\PP^5$ to the product of Pl\"ucker variables which it is indexed by. Explicitly, the columns of $S$ are the exponent vectors of the products of variables: 
\[
S = \ \
\bordermatrix{
& P_2 P_{23} P_{234} & 
  P_2 P_{24} P_{234} & 
  P_3 P_{23} P_{234} &
  P_3 P_{24} P_{234} & 
  P_4 P_{23} P_{234} & 
  P_4 P_{24} P_{234}  \cr
P_2     &1&1& & & & \cr
P_3     & & &1&1& & \cr
P_4     & & & & &1&1\cr
P_{23}  &1& &1& &1& \cr
P_{24}  & &1& &1& &1\cr
P_{234} &1&1&1&1&1&1\cr
}
.
\]
Consider the image of $T$ under the Segre embedding. The ideal of $T$ embedded in $\PP^5$ is the kernel of the following monomial map, whose matrix is the product of $A$ and $S$:
\[
AS = \ \
\bordermatrix{
& P_2 P_{23} P_{234} & 
  P_2 P_{24} P_{234} & 
  P_3 P_{23} P_{234} &
  P_3 P_{24} P_{234} & 
  P_4 P_{23} P_{234} & 
  P_4 P_{24} P_{234}  \cr
x_2 &1&1& & & & \cr
x_3 &1& &2&1&1& \cr
x_4 &1&2&1&2&2&3\cr
y_2 &1&1&1&1&1&1\cr
y_3 &1&1&1&1&1&1\cr
z_2 &1&1&1&1&1&1\cr
}
.
\]
The polytope $\mathcal{P}$ corresponding to this toric variety is the convex hull of the columns of $AS$.
It is easy to see that $\mathcal{P}$ is $2$-dimensional by projecting it to the
coordinates indexed by $x_2, x_3$ and $x_4$. Moreover, $\mathcal{P}$ lives in the $2$-dimensional affine subspace $\{(i,j,k) : i+j+k = 3\}$. See Figure~\ref{fig:richardson_polytope}. 

\begin{figure}
    \centering
    \includegraphics{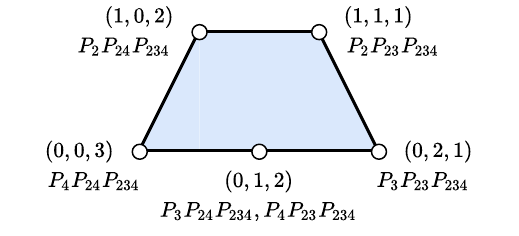}
    \caption{Projection of the polytope $\mathcal P$ of the toric variety in Example~\ref{example:richardson_polytope}. The illustration includes all lattice points of the polytope, which are labelled by the corresponding coordinates of $\PP^5$.}
    \label{fig:richardson_polytope}
\end{figure}
\end{example}

\smallskip
\noindent{\bf Acknowledgement.} NC was supported by the SFB/TRR 191 ``Symplectic structures in Geometry, Algebra and Dynamics''.
He gratefully acknowledges support from the Max Planck Institute for Mathematics in Bonn, and the EPSRC Fellowship EP/R023379/1 who supported his multiple visits to Bristol. 
OC was supported by EPSRC Doctoral Training Partnership 
award EP/N509619/1.
FM was 
supported 
by EPSRC  
Fellowship EP/R023379/1, the BOF grant 
BOF/STA/201909/038, and the FWO  (project no. G023721N and G0F5921N).

%\noindent{\large\bf References}

\bibliographystyle{abbrv} 
\bibliography{Flag-Trop1.bib}

\begin{thebibliography}{10}

\bibitem{Chary_Ollie_Fatemeh}
N.~C. Bonala, O.~Clarke, and F.~Mohammadi.
\newblock Standard monomial theory and toric degenerations of {R}ichardson
  varieties in the {G}rassmannian.
\newblock {\em Journal of Algebraic Combinatorics}, 54(4):1159--1183, 2021.

\bibitem{bossinger2017computing}
L.~Bossinger, S.~Lamboglia, K.~Mincheva, and F.~Mohammadi.
\newblock Computing toric degenerations of flag varieties.
\newblock In {\em Combinatorial algebraic geometry}, pages 247--281. Springer,
  2017.

\bibitem{bossinger2021families}
L.~Bossinger, F.~Mohammadi, A.~N{\'a}jera~Ch{\'a}vez, et~al.
\newblock Families of gröbner degenerations, {G}rassmannians and universal
  cluster algebras.
\newblock {\em SIGMA. Symmetry, Integrability and Geometry: Methods and
  Applications}, 17:059, 2021.

\bibitem{caldero2002toric}
P.~Caldero.
\newblock Toric degenerations of {S}chubert varieties.
\newblock {\em Transformation Groups}, 7(1):51--60, 2002.

\bibitem{clarke2021combinatorial}
O.~Clarke, A.~Higashitani, and F.~Mohammadi.
\newblock Block diagonal polytopes for flag varieties and their combinatorial
  mutations.
\newblock {\em In preparation}, 2021.

\bibitem{clarke2020combinatorial}
O.~Clarke, A.~Higashitani, and F.~Mohammadi.
\newblock Combinatorial mutations and block diagonal polytopes.
\newblock {\em Collectanea Mathematica}, pages 1--31, 2021.

\bibitem{OllieFatemeh}
O.~Clarke and F.~Mohammadi.
\newblock Toric degenerations of {G}rassmannians and {S}chubert varieties from
  matching field tableaux.
\newblock {\em Journal of Algebra}, 559:646--678, 2020.

\bibitem{OllieFatemeh3}
O.~Clarke and F.~Mohammadi.
\newblock Standard monomial theory and toric degenerations of {S}chubert
  varieties from matching field tableaux.
\newblock {\em Journal of Symbolic Computation}, 104:683--723, 2021.

\bibitem{OllieFatemeh2}
O.~Clarke and F.~Mohammadi.
\newblock Toric degenerations of flag varieties from matching field tableaux.
\newblock {\em Journal of Pure and Applied Algebra}, 225(8):106624, 2021.

\bibitem{clarke2022}
O.~Clarke, F.~Mohammadi, and F.~Zaffalon.
\newblock Toric degenerations of partial flag varieties and combinatorial
  mutations of matching field polytopes.
\newblock {\em In preparation}, 2021.

\bibitem{deodhar1985some}
V.~Deodhar.
\newblock On some geometric aspects of {B}ruhat orderings. {I}. {A} finer
  decomposition of {B}ruhat cells.
\newblock {\em Inventiones mathematicae}, 79(3):499--511, 1985.

\bibitem{gonciulea1996degenerations}
N.~Gonciulea and V.~Lakshmibai.
\newblock Degenerations of flag and {S}chubert varieties to toric varieties.
\newblock {\em Transformation Groups}, 1(3):215--248, 1996.

\bibitem{M2}
D.~R. Grayson and M.~E. Stillman.
\newblock Macaulay2, a software system for research in algebraic geometry.
\newblock Available at {http://www.math.uiuc.edu/Macaulay2/}.

\bibitem{hibi1987distributive}
T.~Hibi.
\newblock Distributive lattices, affine semigroup rings and algebras with
  straightening laws.
\newblock In {\em Commutative Algebra and Combinatorics}, pages 93--109, 1987.

\bibitem{hodge1943some}
W.~V.~D. Hodge.
\newblock Some enumerative results in the theory of forms.
\newblock In {\em Mathematical Proceedings of the Cambridge Philosophical
  Society}, volume~39, pages 22--30, 1943.

\bibitem{kim2015richardson}
G.~Kim.
\newblock Richardson varieties in a toric degeneration of the flag variety.
\newblock {\em Thesis (Ph.D.) - University of Michigan. 82 pp. ISBN:
  978-1339-03949-7}, 2015.

\bibitem{KOGAN}
M.~Kogan and E.~Miller.
\newblock Toric degeneration of {S}chubert varieties and {G}elfand-{T}setlin
  polytopes.
\newblock {\em Advances in Mathematics}, 193(1):1--17, 2005.

\bibitem{kreiman2002richardson}
V.~Kreiman and V.~Lakshmibai.
\newblock Richardson varieties in the {G}rassmannian.
\newblock {\em arXiv preprint math/0203278}, 2002.

\bibitem{lakshmibai2003richardson}
V.~Lakshmibai and P.~Littelmann.
\newblock Richardson varieties and equivariant {K}-theory.
\newblock {\em Journal of Algebra}, 260(1):230--260, 2003.

\bibitem{miller2004combinatorial}
E.~Miller and B.~Sturmfels.
\newblock {\em Combinatorial commutative algebra}, volume 227.
\newblock Springer Science \& Business Media, 2004.

\bibitem{richardson1992intersections}
R.~Richardson.
\newblock Intersections of double cosets in algebraic groups.
\newblock {\em Indagationes Mathematicae}, 3(1):69--77, 1992.

\bibitem{seshadri2016introduction}
C.~S. Seshadri.
\newblock {\em Introduction to the theory of standard monomials}, volume~46.
\newblock Springer.

\bibitem{willis2011tableau}
M.~Willis.
\newblock A direct way to find the right key of a semistandard young tableau.
\newblock {\em Annals of Combinatorics}, 17, 10 2011.

\end{thebibliography}

\bigskip
\noindent
\footnotesize {\bf Authors' addresses:}

\medskip

\noindent %Narasimha Chary Bonala\\ 
Ruhr-Universit\"at Bochum,
Fakult\"at f\"ur Mathematik, D-44780 Bochum, Germany
\\
\noindent  E-mail address: {\tt  narasimha.bonala@rub.de}
\medskip

\noindent %Oliver Clarke\\ 
University of Bristol, School of Mathematics,
BS8 1TW, Bristol, UK
\\
\noindent  E-mail address: {\tt oliver.clarke@bristol.ac.uk}

\medskip

\noindent %Fatemeh Mohammadi \\
Department of Mathematics: Algebra and Geometry, Ghent University, 9000 Ghent, Belgium \\
Department of Mathematics and Statistics, 
UiT – The Arctic University of Norway, 9037 Troms\o, Norway
\\ E-mail address: {\tt fatemeh.mohammadi@ugent.be}

\end{document}